\documentclass[a4paper, 11pt]{amsart}
\usepackage{amsmath, amsthm, amscd, amssymb, amsfonts, amsxtra, amssymb, latexsym, mathtools}
\usepackage{enumerate}
\usepackage{verbatim}
\usepackage{textcomp}
\usepackage{float}
\usepackage{pb-diagram}
\usepackage[utf8x]{inputenc}
\usepackage{array}
\usepackage{hyperref}
\hypersetup{colorlinks = true,	allcolors  = blue}

\newcommand{\G}{\Gamma}
\newcommand{\Z}{\mathbb{Z}}

\newcommand{\N}{\mathbb{N}}
\newcommand{\ff}{\mathbb{F}}
\newcommand{\CC}{\mathcal{C}}
\newcommand{\Tr}{\operatorname{Tr}}

\newcommand{\msk}{\medskip}

\newtheorem{thm}{Theorem}[section]
\newtheorem{prop}[thm]{Proposition}

\newtheorem{coro}[thm]{Corollary}

\theoremstyle{definition}
\newtheorem{rem}[thm]{Remark}
\newtheorem{exam}[thm]{Example}

\theoremstyle{remark}

\usepackage{color}

\begin{document} \sloppy 
\numberwithin{equation}{section}
\title[Equienergetic GP-graphs]{Generalized Paley graphs equienergetic \\ with their complements}
\author[R.A.\@ Podest\'a, D.E.\@ Videla]{Ricardo A.\@ Podest\'a, Denis E.\@ Videla}
\dedicatory{\today}
\thanks{Partially supported by CONICET and SECyT-UNC}
%\email{podesta@famaf.unc.edu.ar}
%\email{devidela@famaf.unc.edu.ar}
\address{Ricardo A.\@ Podest\'a, FaMAF -- CIEM (CONICET), Universidad Nacional de C\'ordoba, \newline
	Av.\@ Medina Allende 2144, Ciudad Universitaria, (5000) C\'ordoba, Rep\'ublica Argentina. \newline
%	\address{Ricardo A.\@ Podest\'a -- CIEM, Universidad Nacional de C\'ordoba, CONICET, FAMAF (5000) C\'ordoba, Argentina. 
	{\it E-mail: podesta@famaf.unc.edu.ar}}
\address{Denis E.\@ Videla -- FaMAF -- CIEM (CONICET), Universidad Nacional de C\'ordoba, \newline
	Av.\@ Medina Allende 2144, Ciudad Universitaria, (5000) C\'ordoba, Rep\'ublica Argentina. \newline
	 {\it E-mail: devidela@famaf.unc.edu.ar}}

\begin{abstract}
We consider generalized Paley graphs $\G(k,q)$, generalized Paley sum graphs $\G^+(k,q)$, and their corresponding complements $\bar \G(k,q)$ and $\bar \G^+(k,q)$, for $k=3,4$. 
Denote by $\G=\G^*(k,q)$ either $\G(k,q)$ or $\G^+(k,q)$. 
We compute the spectra of $\G(3,q)$ and $\G(4,q)$ and from them we obtain the spectra of $\G^+(3,q)$ and $\G^+(4,q)$ also. Then we show that, in the non-semiprimitive case, the spectrum of $\G(3,p^{3\ell})$ and $\G(4,p^{4\ell})$  with $p$ prime  
can be recursively obtained, under certain arithmetic conditions, from the spectrum of the graphs $\G(3,p)$ and $\G(4,p)$ for any $\ell \in \N$, respectively. 
Using the spectra of these graphs we give necessary and sufficient conditions on the spectrum of $\G^*(k,q)$ such that $\G^*(k,q)$ and $\bar \G^*(k,q)$ are equienergetic for $k=3,4$. In a previous work we have classified all bipartite regular graphs $\Gamma_{bip}$ and all strongly regular graphs $\Gamma_{srg}$ which are complementary equienergetic, 
i.e.\@ $\{\Gamma_{bip}, \bar{\Gamma}_{bip}\}$ and $\{\Gamma_{srg}, \bar{\Gamma}_{srg}\}$ are equienergetic pairs of graphs. Here we construct infinite pairs of equienergetic non-isospectral regular graphs $\{\Gamma, \bar \Gamma\}$ which are neither bipartite nor strongly regular. 
\end{abstract}

\maketitle

\section{Introduction}
Recently, in \cite{PV5}, we studied regular graphs equienergetic with their complements and we characterized all bipartite graphs and all strongly regular graphs which are equienergetic with their complements. This includes the case of semiprimitive generalized Paley graphs $\G(k,q)$, which are known to be strongly regular. We recall that $\G(k,q)$ is \textit{semiprimitive} if $-1$ is a power of $p$ modulo $k$. Thus, if $\G(k,q)$ is semiprimitive, it is either a classical Paley graph $\G(2,q)$ with $q\equiv 1 \pmod 4$ or else we have $k>2$, $q=p^m$ with $m$ even and $k\mid p^t+1$ for some $t\mid \frac m2$.
Here, by considering generalized Paley graphs of the form $\G(3,q)$ and $\G(4,q)$ which are not semiprimitive, we will produce infinite pairs of complementary equienergetic regular graphs which are neither bipartite nor strongly regular.

\subsubsection*{Spectrum and energy} 
Let $\Gamma$ be a graph of $n$ vertices. The eigenvalues of $\Gamma$ are the eigenvalues $\{\lambda_i\}_{i=1}^n$ of its adjacency matrix. The \textit{spectrum} of $\Gamma$, denoted 
$$Spec(\Gamma) = \{ [\lambda_{i_1}]^{e_1}, \ldots,[\lambda_{i_s}]^{e_{i_s}}\},$$ 
is the set of all the different eigenvalues $\{\lambda_{i_j}\}$ of $\Gamma$ counted with their multiplicities $\{e_{i_j}\}$. 
The \textit{energy} of $\Gamma$ is defined by 
\begin{equation} \label{energy}
E(\Gamma) = \sum_{i=1}^n |\lambda_i| = \sum_{j=1}^s e_{i_j} |\lambda_{i_j}|.
\end{equation}
We refer to the books \cite{BH} or \cite{CDS} for a complete viewpoint of spectral theory of graphs and to \cite{Gu} for a survey on energy of graphs (see also the book \cite{LSG} which contains many open problems related to energy of graphs).

Let $\Gamma_1$ and $\Gamma_2$ be two graphs with the same number of vertices. The graphs are said \textit{isospectral} if $Spec(\Gamma_1) = Spec(\Gamma_2)$ and \textit{equienergetic} if 
$$E(\Gamma_1) = E(\Gamma_2).$$
It is clear by the definitions that isospectrality implies equienergeticity, but the converse does not hold in general.
There are many papers on these problems (see for instance \cite{Ba}, \cite{GPI}, \cite{HX}, \cite{Il}, \cite{PV3}, \cite{PV5}, \cite{Ra} and the references therein).
If a graph $\G$ and its complement $\bar \G$ are equienergetic we will say, as in \cite{RPPAG}, that they are \textit{complementary equienergetic} graphs.
Self-complementary graphs are trivially complementary equienergetic, so the interest is put on non self-complementary graphs.

\subsubsection*{Generalized Paley (sum) graphs}
Let $G$ be a finite abelian group and $S$ a subset of $G$ with $0\notin S$. The \textit{Cayley graph }$X(G,S)$ is the directed graph whose vertex set is $G$ and $v, w \in G$ form a directed edge of $\Gamma$ from $v$ to $w$ if $w-v \in S$. Since $0\notin S$ then $\Gamma$ has no loops. Analogously, the \textit{Cayley sum graph} $X^+(G,S)$ has the same vertex set $G$ but now $v,w\in G$ are connected in $\Gamma$ 
by an arrow from $v$ to $w$ if $v+w \in S$. 
We will use the notation $$X^*(G,S)$$ 
when we want to consider both $X(G,S)$ and $X^+(G,S)$.
Notice that if $S$ is symmetric, that is $-S=S$, then 
$X^*(G,S)$ is an $|S|$-regular simple (undirected without multiple edges) graph. 
However, the graph $X^+(G,S)$ may contain loops. In this case, there is a loop on vertex $x$ provided that  $x+x \in S$. For an excellent survey of spectral properties of general Cayley graphs we refer the reader to \cite{LZ2}. 
In \cite{PV3}, Theorem 2.7, we showed that if $G$ is abelian and $S$ is a symmetric subset of $G$ not containing zero, then $X(G,S)$ and $X^+(G,S)$ are equienergetic graphs. Also, in Proposition 2.10 in \cite{PV3} we give sufficient conditions for $X(G,S)$ and $X^+(G,S)$ to be non-isospectral.

We are interested in \textit{generalized Paley graphs}
\begin{equation} \label{GP} 
\Gamma(k,q) = X(\ff_q,R_k) \qquad \text{with} \qquad R_k = \{x^k: x\in \ff_q^*\},
\end{equation}
that is when $G$ is the finite field $\ff_q$ with $q$ elements and $S$ is the set of nonzero $k$-th powers of $\ff_q$,
and the \textit{generalized Paley sum graph} 
$$\Gamma^+(k,q) = X^+(\ff_q,R_k).$$ 
We will refer to them simply as \textit{GP-graphs} and \textit{GP$^+$-graphs} respectively (or \textit{GP$^*$-graphs} for both indistinctly).
When $k=1$, the graph $\G(1,q)$ is just the complete graph in $q$-vertices $K_q$ and when $k=2$ the graphs $\G(2,q)$ correspond to the classical Paley graphs $P(q)$ if $q \equiv 1 \pmod 4$. The next GP-graphs to consider, aside from general families such as the semiprimitive ones or the Hamming GP-graphs, are those with $k=3,4$, that is $\G(3,q)$ and $\G(4,q)$.
Notice that for $q$ even, we have that $\Gamma^+(k,q)=\Gamma(k,q)$. 
On the other hand, when $q$ is odd, it can be seen that $\Gamma^+(k,q)$ always has loops, since in this case we can find exactly $|R_k|$ elements $x\in \ff_{q}$ such that $x+x\in R_{k}$ (multiplication by $2$ is a bijection in $\ff_{q}$ for $q$ odd). 
In particular, $\Gamma^+(k,q)$ has loops for $k=1,2,3,4$ and $q$ odd.

Generalized Paley graphs have been extensively studied in the few past years. 
Lim and Praeger studied their automorphism groups and characterized all GP-graphs which are Hamming graphs (\cite{LP}). 
Also, Pearce and Praeger characterized those GP-graphs which are Cartesian decomposable (\cite{PP}). Recently, Chi Hoi Yip has studied their clique number (\cite{Y1}, \cite{Y2}, see also %the work of Schneider and Silva 
\cite{SS}).
Both classic Paley graphs and generalized Paley graphs have been used to find linear codes with good decoding properties (\cite{GK}, \cite{KL}, \cite{SL}). 
They can also be seen as particular regular maps in Riemann surfaces 
(\cite{J}, \cite{JW}). 
The number of walks in GP-graphs are related with the number of solutions of 
diagonal equations over finite fields (\cite{V}). Moreover, the diameter of $\G(k,q)$, if it exists, coincides with the Waring number $g(k,q)$ (see \cite{PV6}, \cite{PV7}). 
Under some mild restrictions, the spectrum of GP-graphs determines the weight distribution of their associated irreducible
codes (\cite{PV2}, \cite{PV4}).

\subsubsection*{Cameron's hierarchy}
There is a hierarchy of regularity conditions on graphs due to Cameron (\cite{Ca}). For a non-negative integer $t$ and sets $S_1$ and $S_2$ of at most $t$ vertices, let $\mathcal{C}(t)$ be this graph property:
if the induced subgraphs on $S_1$ and $S_2$ are isomorphic, then the number of vertices joined to every vertex in $S_1$ is equal to the number of vertices joined to every vertex in $S_2$. A graph satisfying property $\mathcal{C}(t)$ is called \textit{$t$-tuple regular}. 
Conditions $\mathcal{C}(t)$ are stronger as $t$ increases. A graph satisfies $\mathcal{C}(1)$ if and only if it is regular and it satisfies $\mathcal{C}(2)$ if and only if it is strongly regular. If a graph satisfies $\mathcal{C}(3)$ then it is the pentagon $C_5$ or it has the parameters of a pseudo Latin square, a negative Latin square or a Smith type graph (see \cite{Ca2}). Up to complements, there are only two known examples of graphs satisfying $\mathcal{C}(4)$ but not $\mathcal{C}(5)$, the Schl\"afli graph and the McLaughlin graph.  Finally, the hierarchy is finite; if a graph satisfies $\mathcal{C}(5)$ then it satisfies $\mathcal{C}(t)$ for any $t$. The only such graphs are $aK_m$ and its complement for $a,m \ge 1$, the pentagon $C_5$ and the $3 \times 3$ square lattice $L(K_{3,3})$ (see \cite{Ca3}).

\subsubsection*{Outline and results}
We now summarize the main results in the paper.
In Section 2, we give the spectrum of $\Gamma^*(3,q)$ and $\Gamma^*(4,q)$ from the spectral relationship between irreducible cyclic codes and GP-graphs (see \cite{PV2}, \cite{PV4}), and the spectral relationship between Cayley graphs and Cayley sum graphs (see \cite{PV3}). Namely, in Theorems \ref{gp3q} and \ref{gp4q} we compute the spectrum of $\Gamma(3,q)$ and $\Gamma(4,q)$ using results in \cite{PV2} while in Theorems \ref{gp3q+} and \ref{gp4q+} we obtain, from the spectrum of $\G(3,q)$ and $\G(4,q)$, the spectrum of the sum graphs $\Gamma^+(3,q)$ and $\Gamma^+(4,q)$, using results in \cite{PV3}.
In the non-semiprimitive case, that is when $p\equiv 1 \pmod k$, the spectra of $\Gamma^*(3,q)$ and $\Gamma^*(4,q)$ are given in terms of certain integer solutions of quadratic diophantine equations. More precisely, solutions of the equations  
\begin{equation} \label{eqs}
p^t=X^2+27Y^2 \quad \text{for $k=3$} \qquad \text{and} \qquad p^{2t}=X^2+4Y^2 \quad \text{for $k=4$},
\end{equation}
with $t\in \N$.

In Section 3 we show that, under certain conditions, the spectrum of $\G(3,q)$ and $\G(4,q)$ in the non-semiprimitive case 
can be obtained recursively from the spectrum of the graphs $\G(3,q')$ and $\G(4,q')$ with $q'<q$. 
More precisely, in Theorem \ref{spectra der 3 gen} we proved that if $p$ is a prime with $p\equiv 1 \pmod 3$ then the spectra of $\G(3,p^{3\ell})$ and $\G(3,p^{3s})$ determine the spectrum of $\G(3,p^{3(t\ell+s})$ for every $\ell \ge 1$ and $0\le s<t$, where $t$ is the minimal integer such that the first equation in \eqref{eqs} has integer solutions $(x,y)$ with $(x,p)=1$.
In Theorem \ref{res cubic} we showed that if $2$ is a cubic residue modulo a prime $p\equiv 1 \pmod 3$ then the spectrum of $\G(3,p)$ determines the spectrum of $\G(3,p^{3\ell})$ for every $\ell \in \N$. Also, in Theorem \ref{spectra der 4} we prove that the spectrum of $\G(4,p)$ always determines the spectrum of $\G(4,p^{4\ell})$ for every $\ell \in \N$. In all these 3 theorems we give the corresponding spectrum explicitly in terms of the integer solution of the corresponding base equation (for instance $p=X^2+27Y^2$ and  $p^2=X^2+4Y^2$ in Theorems \ref{res cubic} and \ref{spectra der 4}, respectively).

In Section 4 we focus on the energy of the graphs $\Gamma^*(3,q)$ and $\Gamma^*(4,q)$. In Proposition \ref{energies} we study their energies, computing them explicitly in the semiprimitive case and giving lower and upper bounds in the non-semiprimitive case.
Then, in Theorem \ref{equien Gp comp} we give conditions on $\Gamma^*(k,q)$ to be equienergetic to its complement $\bar \Gamma^*(k,q)$.
In fact, we show that the graphs $\G(k,q)$ and $\bar \G(k,q)$ are equienergetic if and only if among the non principal eigenvalues of $\G(k,q)$ exactly one of them is positive.

Finally, in Theorems \ref{thm Equien 3} and \ref{thm Equien 4}, by using the results in Sections 3 and 4,
we construct infinite pairs of complementary equienergetic graphs of the form 
$\{\Gamma^*(k,q), \bar \Gamma^*(k,q)\}$ 
for $k=3$ and $4$ and $q=p^m$ which are neither bipartite nor strongly regular. 
This is relevant since we have previously obtained complementary equienergetic pairs of regular graphs which are either bipartite or strongly regular (see Sections 4--7 in \cite{PV5}). 
In fact, we have characterized all such complementary equienergetic pairs of regular graphs. That is to say that, in terms of Cameron's hierarchy, we understand complementary equienergetic graphs in the families $\mathcal{C}(2)$, $\mathcal{C}(3)$ and $\mathcal{C}(5)$; and for bipartirte graphs in $\mathcal{C}(1)$. For general (i.e.\@ non-bipartite) graphs in $\mathcal{C}(1)$ the problem seems to be out of scope. 
The pairs obtained in this work are precisely non-bipartite regular graphs which are not strongly regular, that is non-bipartite graphs in $\mathcal{C}(1) \smallsetminus \mathcal{C}(2)$.

\section{Spectrum} 
Here we will compute the spectrum of the graphs $\G(k,q)$ and $\G^+(k,q)$ for $k=3,4$ using previous results obtained by the authors. 

We begin with the spectrum of the GP-graphs $\G(3,q)$ and $\G(4,q)$, which will follow from an spectral correspondence with the weight distribution of certain cyclic codes $\CC(3,q)$ and $\CC(4,q)$ obtained in \cite{PV2}, that we now recall. Let $p$ be a prime and $q=p^m$ for some $m$. For $k \mid q-1$ consider the $p$-ary irreducible cyclic codes 
\begin{equation} \label{Ckqs}
\mathcal{C}(k,q) = \big \{ c_{\gamma} = \big( \Tr_{q/p}(\gamma\, \omega^{ki}) \big)_{i=0}^{n-1} : \gamma \in \ff_{q} \big \} 
\end{equation}
where $\omega$ is a primitive element of $\ff_{q}$ over $\ff_p$.
These are the codes with zero $\omega^{-k}$ and length
\begin{equation} \label{N and n} 
n = \tfrac {q-1}N \qquad \text{with} \qquad N = \gcd(\tfrac{q-1}{p-1}, k).
\end{equation}
Then, we have the following result.

\begin{thm}[\cite{PV2}] \label{pesoaut}
Let $q=p^m$ with $p$ prime and $k \in \N$ such that $k \mid \frac{q-1}{p-1}$. Put $n=\tfrac{q-1}{k}$.
Let $\G(k,q)$ and $\CC(k,q)$ be as in \eqref{GP} and \eqref{Ckqs}. 
Thus, we have: 
\begin{enumerate}[$(a)$]
\item The eigenvalue $\lambda_{\gamma}$ of $\G(k,q)$ and the weight of $c_{\gamma} \in \CC(k,q)$ are related by the expression
\begin{equation} \label{Pesoaut}
\lambda_{\gamma} = n - \tfrac{p}{p-1} w(c_{\gamma}). 
\end{equation}

\item If $\G(k,q)$ is connected the multiplicity of $\lambda_{\gamma}$ is the frequency $A_{w(c_{\gamma})}$ for all $\gamma\in \ff_q$. In particular, 
the multiplicity of $\lambda_0 = n$ is $A_0=1$. 
\end{enumerate}
\end{thm}

It is known that some structural properties of graphs can be read from the spectrum. For instance, it is a classic result that a regular graph $\G$ is connected if and only if its principal eigenvalue has multiplicity one and that it is bipartite if and only if the spectrum is symmetric. Also, $\G$ is strongly regular if and only if it is connected with 3 different eigenvalues (disregarding multiplicities).

In the next two theorems we give the spectrum of the GP-graphs $\Gamma(3,q)$ and $\G(4,q)$.

\begin{thm} \label{gp3q}
	Let $q=p^{m} \ge 5$ with $p$ prime such that $3\mid \frac{q-1}{p-1}$ and put $n=\frac{q-1}3$. Thus, the graph $\G(3,q)$ 
	is connected with integral spectrum given as follows:
	\begin{enumerate}[$(a)$]
		\item If $p\equiv 1 \pmod 3$ then $m=3t$ for some $t\in \N$ and 
		$$Spec(\G(3,q)) = \big\{ [n]^1, \big[\tfrac{a\sqrt[3]{q}-1}{3}\big]^n, \big[\tfrac{-\frac{1}{2} (a+9b)\sqrt[3]{q}-1}{3}\big]^n, 
		\big[\tfrac{-\frac 12 (a-9b) \sqrt[3]{q}-1}{3}\big]^n \big\} $$
		where $a,b$ are integers uniquely determined by 
		\begin{equation} \label{ab27}
		4\sqrt[3]{q}=a^2+27b^2, \qquad a\equiv 1 \pmod 3 \qquad \text{and} \qquad (a,p)=1.
		\end{equation} 
		%$$ 4\sqrt[3]{q}=a^2+27b^2, \qquad a\equiv 1 \pmod 3 \qquad \text{and} \qquad (a,p)=1.$$
		
		\item If $p\equiv 2 \pmod 3$ then $m=2t$ for some $t\in \N$ and 
		$$Spec(\G(3,q)) = \begin{cases} 
		\big\{ [n]^1, \big[\tfrac{\sqrt{q}-1}{3}\big]^{2n}, \big[\tfrac{-2\sqrt{q}-1}{3}\big]^n \big\} & 
		\qquad \text{for $m\equiv 0 \pmod 4$}, \\[3mm]
		\big\{ [n]^1, \big[\tfrac{2\sqrt{q}-1}{3}\big]^{n}, \big[\tfrac{-\sqrt{q}-1}{3}\big]^{2n} \big\} & 
		\qquad \text{for $m\equiv 2 \pmod 4$}.
		\end{cases}$$
		In particular, $\G(3,q)$ is a strongly regular graph in this case. 
	\end{enumerate}
\end{thm}

\begin{proof}
	Let $q=p^m$. First note that condition $3\mid \frac{q-1}{p-1}=p^{m-1}+\cdots +p+1$ implies that 
	$m=3t$ if $p\equiv 1 \pmod 3$ and $m=2t$ if $p\equiv 2 \pmod 3$. 
	
	We will apply Theorem \ref{pesoaut} to the code $\CC(3,q)$. The spectrum of $\CC(3,q)$ is given in Theorems 19 and 20 in  
	\cite{DY}, with different notations ($r$ for our $q$, $N$ for our $k$, etc). 
	By $(a)$ in Theorem \ref{pesoaut}, the eigenvalues of $\G(3,q)$ are given by 
	\begin{equation} \label{eigen3} 
	\lambda_i = \tfrac{q-1}3 - \tfrac{p}{p-1} w_i
	\end{equation} 
	where $w_i$ are the weights of $\CC(3,q)$. 
	
	If $p\equiv 1 \pmod 3$, by Theorem 19 in \cite{DY}, the four weights of $\CC(3,q)$ are $w_0=0$,
	\begin{equation} \label{sp1} 
	w_1= \tfrac{(p-1)(q-a\sqrt[3]{q})}{3p}, \qquad w_2= \tfrac{(p-1)(q+\frac 12 (a+9b) \sqrt[3]{q})}{3p}, 
	\qquad w_3= \tfrac{(p-1)(q+\frac 12 (a-9b) \sqrt[3]{q})}{3p},
	\end{equation}
	with frequencies $A_0=1$ and $A_1=A_2=A_3=\frac{q-1}3$; 
	where $a$ and $b$ are the only integers satisfying $4\sqrt[3]{q}=a^2+27b^2$, $a\equiv 1 \pmod 3$ and $(a,p)=1$.
	
	On the other hand, if $p\equiv 2 \pmod 3$, by Theorem 20 in \cite{DY}, the three weights of $\CC(3,q)$ are 
	\begin{equation} \label{sp2} 
	w_0=0, \qquad w_1= \tfrac{(p-1)(q-\sqrt{q})}{3p}, \qquad w_2= \tfrac{(p-1)(q+2\sqrt{q})}{3p},
	\end{equation}
	with frequencies $A_0=1$, $A_1=\tfrac{2(q-1)}3$ and $A_2=\tfrac{q-1}3$ 
%	$$A_0=1, \quad A_1=\tfrac{2(q-1)}3 \quad \text{and} \quad A_2=\tfrac{q-1}3$$ 
	if $m\equiv 0 \pmod 4$ while 
	\begin{equation} \label{sp3} 
	w_0=0, \qquad w_1= \tfrac{(p-1)(q-2\sqrt{q})}{3p}, \qquad w_2= \tfrac{(p-1)(q+\sqrt{q})}{3p},
	\end{equation} 
	with frequencies $A_0=1$, $A_1=\tfrac{(q-1)}3$ and $A_2=\tfrac{2(q-1)}3$ 
	%$$A_0=1, \quad A_1=\tfrac{(q-1)}3 \quad \text{and} A_2=\tfrac{2(q-1)}3$$ 
	if $m\equiv 2 \pmod 4$. 
	By introducing \eqref{sp1}, \eqref{sp2} and \eqref{sp3} in \eqref{eigen3}, we get the eigenvalues in ($a$) and $(b$) of the statement. 
	
	\smallskip 
	To compute the multiplicities, we use $(b)$ of Theorem \ref{pesoaut}. The hypothesis that $\G(3,q)$ is connected is equivalent to  
	the fact that $n=\frac{q-1}3$ is a primitive divisor of $q-1$ (see the Introduction, after \eqref{GP}). 
	We now show that this is always the case for $q\ge 5$. 
	
	Suppose that $p\ge 5$. Then, we have that $p^{m-1}-1 \le \frac{q-1}3$, since this inequality is equivalent to $3p^{m-1}-3 \le p^m-1$ 
	which holds for $p\ge 5$. This implies that $n$ is greater that $p^a-1$ for all $a<m$ and, hence, $n$ is a primitive divisor of $q-1$.
	Now, if $p=2$ we only have to check that $n$ does not divide $2^{m-1}-1$, since $n> 2^{m-2}-1$. Notice that $n\mid 2^{m-1}-1$ if and only if $2^{m}-1\mid 3(2^{m-1}-1)$, which can only happen if $m=2$. If $m>2$ we have 
	$3(2^{m-1}-1) \equiv 2^{m-1}-2 \not \equiv 0 \pmod 2^m-1$
	as we wanted. The prime $p=3$ is excluded by hypothesis.
	
Thus, by part $(b)$ of Theorem \ref{pesoaut}, the multiplicities of the eigenvalues of $\G(3,q)$ are the frequencies of the weights of 
$\CC(3,q)$, and we are done.
\end{proof}

Note that in case $(a)$ of the previous theorem, the non-principal eigenvalues 
are of the form 
$$\lambda = \tfrac{\alpha \sqrt[3]q-1}{3} \qquad \text{with} \qquad \alpha = a, \: \tfrac 12(a-9b), \: \tfrac 12(a+9b)$$ 
where $(a,b)$ are the solutions of $4 \sqrt[3]q = X^2+27Y^2$ with $a\equiv 1 \pmod 3$ and $(a,p)=1$.
Furthermore, since $\sqrt{p^{4t+2}}=p\sqrt{p^{4t}}$, for $p=2$ we have a relation between the eigenvalues of $\Gamma(3,2^{2t})$, $\Gamma(3,2^{2t+2})$ and $\Gamma(3,2^{2t+4})$. Namely, 
$$\lambda_3(\Gamma(3,3^{2t})) = \lambda_3(\Gamma(3,3^{2t+2}))
\qquad \text{and} \qquad 
\lambda_2(\Gamma(3,3^{2t+2})) = \lambda_2(\Gamma(3,3^{2t+4})),$$ 
where $\lambda_1 >\lambda_2 >\lambda_3$.

\begin{exam} \label{ej g3g4}
($i$) Let $p=7$ and $m=3$, hence $q=7^3=343$. Since $p\equiv 1 \pmod 3$, we must find $a,b \in \Z$ such that
	$28=a^2+27b^2$, $a\equiv 1 \pmod 3$ and $(a,7)=1$. Clearly $a=b=1$ satisfy these conditions. By Theorem~\ref{gp3q} ($i$) we have  
	$Spec(\G(3,7^3)) = \big\{ [114]^1, [9]^{114}, [2]^{114}, [-12]^{114} \big\}$.

($ii$)
Let $p=2$. By Theorem~\ref{gp3q} ($ii$) we have for instance $Spec(\G(3,2^4)) = \big\{ [5]^1, [1]^{10}, [-2]^{5} \big\}$, 
$Spec(\G(3,2^6)) = \big\{ [21]^1, [5]^{21}, [-3]^{42} \big\}$ and 
$Spec(\G(3,2^8)) = \big\{ [85]^1, [5]^{170}, [-3]^{85} \big\}$.
For $p=5$ we have $Spec(\G(3,5^2)) = \big\{ [8]^1, [3]^{8}, [-2]^{16} \big\}$ and $Spec(\G(3,5^4)) = \big\{ [208]^1, [24]^{416}, [-17]^{208} \big\}$.  
	\hfill $\lozenge$
\end{exam}

\begin{thm}%[\cite{PV4}] 
	\label{gp4q}
	Let $q=p^m$ with $p$ prime such that $4\mid \frac{q-1}{p-1} $ and $q\ge 5$ with $q\ne 9$.  
	Let $n=\frac{q-1}4$. 
	Thus, the graph $\G(4,q)$ is connected with 
	integral spectrum given as follows:
	\begin{enumerate}[$(a)$]
		\item If $p\equiv 1 \pmod 4$ then $m=4t$ for some $t\in \N$ and
		$$Spec(\G(4,q)) = \big\{ [n]^1, \big[\tfrac{\sqrt{q} + 4d\sqrt[4]{q}-1}{4}\big]^n, \big[\tfrac{\sqrt{q} - 4d\sqrt[4]{q}-1}{4}\big]^n, 
		\big[\tfrac{-\sqrt{q} + 2c \sqrt[4]{q}-1}{4}\big]^n, \big[\tfrac{-\sqrt{q} - 2c \sqrt[4]{q}-1}{4}\big]^n \big\}$$
		where $c,d$ are integers uniquely determined by 
\begin{equation} \label{c4d}
		\sqrt{q}=c^2+4d^2, \qquad c\equiv 1 \pmod 4 \qquad \text{and} \qquad (c,p)=1.
\end{equation} 
		
		\item If $p\equiv 3 \pmod 4$ then $m=2t$ for some $t\in \N$ and 
		$$Spec(\G(4,q)) = \begin{cases} 
		\big\{ [n]^1, \big[\tfrac{\sqrt{q}-1}{4}\big]^{3n}, \big[\tfrac{-3\sqrt{q}-1}{4}\big]^n \big \} & 
		\qquad \text{for $m\equiv 0 \pmod 4$}, \\[3mm]
		\big\{ [n]^1, \big[\tfrac{3\sqrt{q}-1}{4}\big]^{n}, \big[\tfrac{-\sqrt{q}-1}{4}\big]^{3n} \big\} & 
		\qquad \text{for $m\equiv 2 \pmod 4$}.
		\end{cases}$$
		In particular, $\G(4,q)$ is a strongly regular graph in this case. 
	\end{enumerate}
\end{thm}

\begin{proof}
	The proof is similar to the one of Theorem \ref{gp3q}. We apply Theorem \ref{pesoaut} to the code $\CC(4,q)$ since the spectrum of this code is given in Theorem 21 in \cite{DY}. Thus, we skip the details and only show that if $q\ge 5$ with $q\ne 9$, then $\frac{q-1}4$ is a primitive divisor of $q-1$ and, hence, $\G(4,q)$ is connected. 
	
	Suppose that $p\ge 5$. Then, we have that $p^{m-1}-1 \le \frac{q-1}4$ since this inequality is equivalent to $4p^{m-1}-4 \le p^m-1$, which is true because $p\ge 5$. This implies that $n$ is greater that $p^a-1$ for all $a<m$ and hence $n$ is a primitive divisor of $q-1$.
	Now, if $p=3$ we only have to check that $n$ does not divide $3^{m-1}-1$, since $n> 3^{m-2}-1$. Notice that $n\mid 3^{m-1}-1$ 
	if and only if $3^{m}-1\mid 4(3^{m-1}-1)$, which can only happen if $m=2$. If $m>2$ in this case 
	$4(3^{m-1}-1) \equiv 3^{m-1}-3 \not\equiv 0 \pmod{3^m-1}$
	as we wanted. The prime $p=2$ is excluded by hypothesis.
\end{proof}

In case $(a)$ of the previous theorem, the non-principal eigenvalues 
are of the form 
$$\lambda=\tfrac{\alpha \sqrt[4]q-1}{4} \qquad \text{with} \qquad \alpha = \sqrt q \pm 4d, \: -\sqrt q \pm 2c$$ 
where $(c,d)$ are integer solutions of $4 \sqrt[3]q = X^2+27Y^2$ with $a\equiv 1 \pmod 3$ and $(a,p)=1$. 
Also, since $\sqrt{p^{2t+2}} = p\sqrt{p^{t}}$, for $p=3$ we have the relations $$\lambda_3(\Gamma(3,3^{2t})) = \lambda_3(\Gamma(3,3^{2t+2})) \qquad \text{and} \qquad 
\lambda_2(\Gamma(3,3^{2t+2})) = \lambda_2(\Gamma(3,3^{2t+4}))$$ 
between the eigenvalues of $\Gamma(3,3^{2t})$, $\Gamma(3,3^{2t+2})$ and $\Gamma(3,3^{2t+4})$, 
where $\lambda_1 >\lambda_2 >\lambda_3$.

\begin{rem} \label{solutions eq}
We now make two observations relative to Theorems \ref{gp3q} and \ref{gp4q}. 

\noindent ($i$)
In parts $(b)$ of these theorems the graphs $\G(3,q)$ and $\G(4,q)$ are semiprimitive (see the Introduction). 
The spectrum of semiprimitive generalized Paley graphs $\G(k,q)$ was studied and computed in \cite{PV2}.
Thus, parts $(b)$ in these theorems can be obtained as particular cases of Theorem 3.3 in \cite{PV2} with $k=3,4$.

\noindent $(ii)$ By the theorems, % \ref{gp3q} and \ref{gp4q}, 
it is implicit that the equations 
$4p^{t}=X^2+27 Y^{2}$ for $p\equiv1 \pmod{3}$  and % \qquad \text{and} \qquad 
$p^{2t}=X^2+4Y^2$ for $p\equiv 1 \pmod{4}$ always have integer solutions $(x,y)$ with $(x,p)=1$, where $x\equiv 1 \pmod{3}$ for the first equation and $x\equiv 1 \pmod{4}$ for the second one. This is known from number theory results.
\end{rem}

\begin{exam}
($i$) Let $q=5^4=625$, that is $p=5$ and $m=4$. Since $p\equiv 1 \pmod 4$, we have to find integers  $c,d$ such that
	$25=c^2+4d^2$, $c\equiv 1 \pmod 4$ and $(c,5)=1$. One can check that $(c,d)=(-3,2)$ satisfy these conditions and hence by ($i$) in Theorem~\ref{gp4q}, the spectrum of $\G(4,625)$ is given by 
	$$Spec(\G(4,625)) = \big\{ [156]^1, [16]^{156}, [1]^{156}, [-4]^{156}, [-14]^{156} \big\}.$$

($ii$) Let $p=3$. By ($b$) in Theorem~\ref{gp4q} we have 
$Spec(\G(4,3^4)) = \big\{ [20]^1, [2]^{60}, [-7]^{20} \big\}$, 
$Spec(\G(4,3^6)) = \big\{ [182]^1, [20]^{182}, [-7]^{546} \big\}$ and 
$Spec(\G(4,3^8)) = \big\{ [1640]^1, [20]^{4920}, [-61]^{1640} \big\}$.  
	\hfill $\lozenge$
\end{exam}

%In $(b)$ of Theorems \ref{gp3q} and \ref{gp4q}, the graphs $\G(3,q)$ and $\G(4,q)$ are semiprimitive (see the Introduction). 
%The spectrum of semiprimitive generalized Paley graphs $\G(k,q)$ was studied and computed in \cite{PV2}.

\begin{rem} \label{ab eig}
Notice that from items ($a$) of the previous two theorems, if we denote by $\lambda_1,\ldots,\lambda_k$ the non-principal eigenvalues of $\G(k,q)$, with $k=3$ or $4$, in the order listed in the theorems, we then have that
$a= \tfrac{3\lambda_1+1}{p^t}$, $b= \tfrac{\lambda_3 -\lambda_2}{3p^t}$,
and $c= \tfrac{\lambda_3 - \lambda_4}{p^t}$, $d= \tfrac{\lambda_1 -\lambda_2}{2p^t}$, 
respectively.
\end{rem}

We now give the spectrum of the sum GP-graphs $\G^+(3,q)$ and $\G^+(4,q)$ for $q$ odd, since for $q$ even $\G^+(3,q)=\G(3,q)$ and $\G^+(4,q)=\G(4,q)$. The spectra of $\G^+(3,q)$ and $\G^+(4,q)$ will be obtained from the corresponding ones of $\G(3,q)$ and $\G(4,q)$ using the same techniques as in \cite{PV3}. 

\begin{thm}%[\cite{PV4}] 
	\label{gp3q+}
	Let $q=p^{m}$ with $p$ an odd prime and $m \in \N$ such that $3\mid \frac{q-1}{p-1}$ and $q\ge 5$ and put $n=\frac{q-1}3$. 
	Thus, the graph $\G^+(3,q)$ is connected with integral %almost symmetric 
	spectrum given by:
	\begin{enumerate}[$(a)$]
		\item If $p\equiv 1 \pmod 3$ then $m=3t$ for some $t\in \N$ and 
		$$Spec(\G^+(3,q)) = \big\{ [n]^1, \big[\pm \tfrac{a\sqrt[3]{q}-1}{3}\big]^{\frac n2}, \big[\pm \tfrac{-\frac{1}{2} (a+9b)\sqrt[3]{q}-1}{3}\big]^{\frac n2}, 
		\big[\pm \tfrac{-\frac 12 (a-9b) \sqrt[3]{q}-1}{3}\big]^{\frac n2} \big\} $$
		where $a,b$ are integers uniquely determined by conditions \eqref{ab27}. \msk

		\item If $p\equiv 2 \pmod 3$ then $m=2t$ for some $t\in \N$ and 
		$$Spec(\G^+(3,q)) = \begin{cases} 
		\big\{ [n]^1, \big[\pm \tfrac{\sqrt{q}-1}{3}\big]^{n}, \big[\pm \tfrac{-2\sqrt{q}-1}{3}\big]^{\frac n2} \big\} & 
		\qquad \text{for $m\equiv 0 \pmod 4$}, \\[2mm]
		\big\{ [n]^1, \big[\pm \tfrac{2\sqrt{q}-1}{3}\big]^{\frac n2}, \big[\pm \tfrac{-\sqrt{q}-1}{3}\big]^{n} \big\} & 
		\qquad \text{for $m\equiv 2 \pmod 4$}.
		\end{cases}$$
	\end{enumerate}
\end{thm}

\begin{proof}
	It is a direct consequence of Proposition 2.10 in \cite{PV3} and Theorem \ref{gp3q}, 
	since $\ff_q$ has no trivial real characters for $q$ odd.
\end{proof}

In the same way, we have the following result.
\begin{thm}%[\cite{PV4}] 
	\label{gp4q+}
	Let $q=p^m$ with $p$ an odd prime and $m \in \N$ such that $4\mid \frac{q-1}{p-1} $ and $q\ge 5$ with $q\ne 9$ and put $n=\frac{q-1}4$.
	Thus, the graph $\G^+(4,q)$ is connected with integral 
	spectrum given by:
	\begin{enumerate}[$(a)$]
		\item If $p\equiv 1 \pmod 4$ then $m=4t$ for some $t\in \N$ and
		\small{$$Spec(\G^+(4,q)) = \big\{ [n]^1, \big[\pm \tfrac{\sqrt{q} + 4d\sqrt[4]{q}-1}{4}\big]^{\frac n2}, \big[\pm \tfrac{\sqrt{q} - 4d\sqrt[4]{q}-1}{4}\big]^{\frac n2}, \big[\pm \tfrac{-\sqrt{q} + 2c \sqrt[4]{q}-1}{4}\big]^{\frac n2}, 
			\big[\pm \tfrac{-\sqrt{q} - 2c \sqrt[4]{q}-1}{4}\big]^{\frac n2} \big\}$$}
		where $c,d$ are integers uniquely determined by conditions \eqref{c4d}. \msk

		\item If $p\equiv 3 \pmod 4$ then $m=2t$ for some $t\in \N$ and 
		$$Spec(\G^+(4,q)) = \begin{cases} 
		\big\{ [n]^1, \big[\pm \tfrac{\sqrt{q}-1}{4}\big]^{\frac{3n}2}, \big[\mp \tfrac{3\sqrt{q}+1}{4}\big]^{\frac n2} \big \} & 
		\qquad \text{for $m\equiv 0 \pmod 4$}, \\[2mm]
		\big\{ [n]^1, \big[\mp \tfrac{\sqrt{q}+1}{4}\big]^{\frac{3n}2}, \big[\pm \tfrac{3\sqrt{q}-1}{4}\big]^{\frac n2} \big \} & 
		\qquad \text{for $m\equiv 2 \pmod 4$}.
		\end{cases}$$
	\end{enumerate}
\end{thm}

\begin{proof}
Again we use Proposition 2.10 in \cite{PV3} and Theorem \ref{gp3q}. 
\end{proof}

As in Remark \ref{ab eig}, we can put the integers $a,b,c,d$ appearing in Theorems \ref{gp3q+} and \ref{gp4q+} in term of the eigenvalues.

Of course, from the spectrum of $\G=\G(k,q)$ for $k=3,4$ one can obtain the spectrum of the complementary graph $\bar \G$ and the spectra  of the associated line graphs $\mathcal{L}(\G)$, $\mathcal{L}(\bar \G)$ and of their complements $\overline{\mathcal{L}(\G)}$ and
 $\overline{\mathcal{L}(\bar \G)}$.  

We finish the section with some comments on Theorems \ref{gp3q}, \ref{gp4q}, \ref{gp3q+} and \ref{gp4q+} 
	%the theorems of the section.
which will be referred to simply as `the theorems'. 
\begin{rem} \label{equinoiso} %\label{rem gp34q}
%We will refer to  simple as .
\noindent ($i$) The integers $b$ and $d$ in `the theorems' are determined up to sign. However, by symmetry, 
the eigenvalues in these theorems are not affected by these choices of sign.  
%\end{rem}

\noindent ($ii$) From `the theorems' we know that all the graphs $\G^*(3,q)$, $\G^*(4,q)$ are connected and non-bipartite. The semiprimitive graphs $\G(k,q)$ with $k=3,4$, i.e.\@ those in items ($b$) of Theorems~\ref{gp3q} and \ref{gp4q}, are strongly regular while the non-semiprimitive graphs (all other graphs in these four theorems) are non strongly regular. 

%\begin{rem} 
\noindent ($iii$) By `the theorems', 
$\{\G(3,q),\G^+(3,q)\}$ and $\{\G(4,q),\G^+(4,q)\}$ are pairs of equienergetic non-isospectral graphs 
%(by Theorems \ref{gp3q}--\ref{gp4q+}), 
which are non bipartite. In the non semiprimitive case ($q\equiv 1 \pmod k$ with $k=3,4$) 
they are also non strongly regular graphs.
\end{rem}

\section{Derived spectrum on field extensions} 
In this section we show that, under certain mild hypothesis, 
we can give the spectrum of $\G(k,p^{k\ell})$ for $k=3,4$, with $p$ any prime number congruent to 1 mod $k$ and $\ell>1$, in the non-semiprimitive case (i.e.\@ for $p\equiv 1 \pmod k$) from the spectra of $\G(k,p^{k})$.
More precisely, by ($a$) in Theorem \ref{gp3q}, if $q=p^{3\ell}$ and $p\equiv 1 \pmod 3$ then the eigenvalues of $\G(3,q)$ and $\G(3,q^n)$ are respectively given in terms of certain integer solutions $(a,b)$ and $(a_n,b_n)$ of the equations 
$$4q=X^2+27Y^2 \qquad \text{and} \qquad 4q^n=X^2+27Y^2.$$ 
We will show that we can recursively obtain integer solutions $(a_n,b_n)$ of $4q^n=X^2+27Y^2$ of the required form in infinitely many field extensions $\ff_{q^n}$ of $\ff_q$ from an initial integer solution $(a,b)$ of $q=x^2+27y^2$ of the required form. Similarly for $\G(4,q)$ and $\G(4,q^n)$ by using Theorem \ref{gp4q}, although this case is much easier.  

For $k=3$, we can give a more general result. Equations of the form $4p^{3t}=X^2+27Y^2$ always have integers solutions (see Remark \ref{solutions eq}), but nothing can be said about the solutions of $p^{3t}=X^2+27Y^2$. However, assuming that there is a minimal $t$ such that $p^{3t}=X^2+27Y^2$ has an integer solution we can provide the spectrum of $\G(3,p^{3(t\ell+s)})$ for any $\ell \in \N$ and $0\le s<t$.

\begin{thm} \label{spectra der 3 gen}
Let $p$ be a prime with $p\equiv 1 \pmod{3}$.
If there is a minimal $t\in \N$ such that %the equation 
\begin{equation} \label{eq pt=x227y2}
p^t=X^2+27 Y^2
\end{equation}
has integral solutions $x,y \in \Z$ with $(x,p)=1$, 
then the spectra of $\G(3,p^{3(t\ell+s)})$, with $\ell\ge 1$ and $0\le s<t$, 
is determined by the spectra of the GP-graphs $\G(3,p^{3t})$ and $\G(3,p^{3s})$.
\end{thm}

\begin{proof}
Let $t$ be minimal in $\N$ such that \eqref{eq pt=x227y2} has an integral solution $x,y$ with $(x,p)=1$. Notice that if $(x,y)$ is a solution of \eqref{eq pt=x227y2} then $(x,-y)$ and $(-x, \pm y)$ are also solutions. Also, from \eqref{eq pt=x227y2} we have that $x^2\equiv 1 \pmod 3$ since $p\equiv 1 \pmod 3$ and hence $x\equiv \pm 1 \pmod 3$. Thus, we will choose one solution $(x_0,y_0)$, with $x_0 \in \{\pm x\}$ and $y_0 \in \{\pm y\}$, such that $x_0\equiv 1 \pmod 3$.

Considering the complex number 
\begin{equation} \label{zxy}
z_{x,y}:=x + 3\sqrt{3} i y, 
\end{equation}
we have that $\| z_{x,y} \|^2 = x^2+27y^2 = p^t$ and hence 
\begin{equation} \label{ptl}
p^{t\ell} = \|z_{x,y}\|^{2 \ell} = \|z_{x,y}^{\ell}\|^2
\end{equation}
for any $\ell \in \N$. Now, we will express $z_{x,y}^\ell$ in the form given in \eqref{zxy}. 
For any $\ell \in \N$ put 
$$z_{x,y}^{\ell} := z_{x_{\ell-1},y_{\ell-1}} = x_{\ell-1}+ 3\sqrt{3}i y_{\ell-1}$$
where $ z_{x,y}^1=z_{x,y}$ and $x_0=x$, $y_0=y$.
For instance, $x_1+3\sqrt 3 i y_1 = z_{x,y}^2 = (x^2-27y^2)+3\sqrt 3 i (2xy)$ so $x_1=x^2-27y^2$ and $y_1=2xy$.
By the relation $z_{x,y}^{\ell+1} = z_{x,y}z_{x,y}^\ell$, one sees that
the sequence $\{(x_{\ell},y_{\ell})\}_{\ell \in \mathbb{N}_0}$ is thus recursively defined as follows: let $x_0=x$, $y_0=y$ and 
for any $\ell>0$ take 
\begin{equation} \label{recursion t1}
x_{\ell} = x_0 x_{\ell-1} - 27 y_0 y_{\ell-1} \qquad \text{and} \qquad y_{\ell} = x_0 y_{\ell-1} + x_{\ell-1}y_0.
\end{equation}

Note that if $\ell\ge1$, starting from $y_{\ell}=x_0 y_{\ell-1} + x_{\ell-1}y_0$ and 
changing the $y_j$'s recursively we obtain the following equality
\begin{equation}\label{y ies t2}
y_{\ell}= y_0 \big(x_0^\ell + \sum_{i=0}^{\ell-1} x_{i} x_{0}^{\ell-1-i} \big).
\end{equation}
So, from the second equation in \eqref{recursion t1} and \eqref{y ies t2} we get  
\begin{equation} \label{x ies t2}
x_{\ell-1} = \tfrac{1}{y_0}(y_{\ell}-x_0 y_{\ell-1}) = %{\blue x_0^\ell} + 
\sum_{i=0}^{\ell-1} x_{i} x_{0}^{\ell-1-i} - x_0 %\Big( x_0^{\ell-1} + 
\sum_{i=0}^{\ell-2} x_{i} x_{0}^{\ell-2-i}. % \Big).
\end{equation}
From the first equation in \eqref{recursion t1}, by using \eqref{y ies t2} and \eqref{x ies t2} and the fact that $p^t=x^{2}+27 y^2$, for all $\ell>1$ we obtain that
\begin{equation}\label{x ies t3}
x_{\ell}= x_0 \sum_{i=0}^{\ell-1} x_{i} x_{0}^{\ell-1-i}-p^t \sum_{i=0}^{\ell-2} x_{i} x_{0}^{\ell-2-i} 
- 27y_0^2x_0^{\ell-1}. 
\end{equation}

\noindent
\textit{Claim 1:} $x_{\ell} \equiv 1\pmod{3}$ and $(x_{\ell},p)=1$ for all $\ell\in \mathbb{N}_0$.

\smallskip

\noindent
\textit{Proof of the claim:} Clearly, $x_0\equiv 1\pmod 3$ since we choose $x$ satisfying this property. 
From the first equation in \eqref{recursion t1} we have that $x_\ell \equiv x_{\ell-1} \pmod 3$ and hence the first statement 
follows by induction. 

On the other hand, we have that $(x_0,p)=1$ by hypothesis.
Notice that $x_1=x^2-27 y^2$ and $y_1=2xy$. 
By taking into account that $p^t=x^2+27 y^2$, 
we obtain that $x_1=2 x^2 -p^t\equiv 2x^2 \pmod{p}$. 
Since $p>3$ is prime and $(x,p)=1$, we obtain that $x_1\not\equiv 0\pmod{p}$  and thus $(x_1,p)=1$.
We now prove that $(x_\ell,p)=1$ for any $\ell \ge 2$ by contradiction.

Suppose that the second statement of the claim is false, so there exists a minimum $L>1$ such that $p\mid x_L$, that is $x_{L}\equiv 0 \pmod{p}$. 
By \eqref{x ies t3}, we obtain that 
$$x_0 \sum_{i=0}^{L-1} x_{i} x_{0}^{L-1-i} - 27y_0^2x_0^{L-1} \equiv x_{L} \equiv 0\pmod{p},$$
and using that $27y_0^2= p^t-x_0^2$ we get
$$x_0 \sum_{i=0}^{L-1} x_{i} x_{0}^{L-1-i} + x_0^{L+1} = x_0 \big(x_0^L + \sum_{i=0}^{L-1} x_{i} x_{0}^{L-1-i} \big) 
\equiv 0\pmod{p}.$$
Since $(x_0,p)=1$, we have that 
\begin{equation} \label{cong 1}
x_0^L + \sum\limits_{i=0}^{L-1} x_{i} x_{0}^{L-1-i}\equiv 0\pmod{p}.
\end{equation}
Notice that 
$$\sum_{i=0}^{L-1} x_{i} x_{0}^{L-1-i}= x_{L-1}+\sum_{i=0}^{L-2} x_{i} x_{0}^{L-1-i}.$$
By applying \eqref{x ies t3} with $\ell=L-1$ we arrive at
$$x_{L-1}\equiv x_0 \sum_{i=0}^{L-2} x_{i} x_{0}^{L-2-i} - 27y_0^2x_0^{L-2} \equiv \sum_{i=0}^{L-2} x_{i} x_{0}^{L-1-i} +  x_{0}^{L} \pmod{p}$$
where we again used that $27y_0^2 = p^t-x_0^2$.
Thus, we have that
$$2 x_{L-1}\equiv x_{L-1} + \sum_{i=0}^{L-2} x_{i} x_{0}^{L-1-i} + x_{0}^{L} \equiv x_{0}^{L} + \sum_{i=0}^{L-1} x_{i} x_{0}^{L-1-i} \equiv 0\pmod{p},$$
by \eqref{cong 1}. Hence $x_{L-1}\equiv 0 \pmod{p}$ since $(2,p)=1$, which contradicts the minimality of $L$. Therefore $(x_{\ell},p)=1$ for all $\ell\in \mathbb{N}_0$. This proves the claim. \hfill $\diamond$

Notice that by \eqref{ptl} and Claim 1, we have obtained a double sequence of integers $\{(x_{\ell},y_{\ell})\}_{\ell\in \mathbb{N}_0}$ such that
$$p^{t(\ell+1)}= x_{\ell}^2+27 y_{\ell}^2 \qquad \text{with} \qquad  
x_\ell\equiv 1\pmod{3} \quad \text{and} \quad (x_{\ell},p)=1.$$
We now seek for solutions of the equation $4p^{t\ell}=X^2+27Y^2$. Since $p>3$, by defining 
$$a_{\ell,0}=-2x_{\ell-1} \qquad \text{and} \qquad b_{\ell,0}=-2 y_{\ell-1}$$ 
for $\ell>0$ we get a sequence of integers $\{(a_{\ell,0},b_{\ell,0})\}_{\ell\in\mathbb{N}}$ satisfying
$$4 p^{t\ell}= a_{\ell,0}^{2}+27 b_{\ell,0}^2 \qquad \text{with} \qquad (a_{\ell,0},p)=1 \quad \text{and} \quad  a_{\ell,0} \equiv 1\pmod{3},$$
which are just the conditions \eqref{ab27} in Theorem \ref{gp3q}. Therefore, if we put $n_{\ell,0}=\frac{p^{3t\ell}-1}{3}$, then the spectrum of $\G(3,p^{3t\ell})$ is given by
\begin{equation} \label{spec p3tl}
\mathrm{Spec}\,\G(3,p^{3t\ell})=\big\{ [n_{\ell,0}]^1, \big[ \tfrac{a_{\ell,0} p^{t\ell}-1}{3}\big]^{n_{\ell,0}}, \big[ \tfrac{-\frac{1}{2} (a_{\ell,0}+9b_{\ell,0})p^{t\ell}-1}{3}\big]^{n_{\ell,0}}, 
\big[ \tfrac{-\frac 12 (a_{\ell,0}-9b_{\ell,0}) p^{t\ell}-1}{3}\big]^{n_{\ell,0}} \big\}.
\end{equation}
Moreover, the sequence $\{(a_{\ell,0},b_{\ell,0})\}_{\ell\in\mathbb{N}}$ satisfies the recursions
\begin{equation} \label{recursions}
a_{\ell+1,0}=x_0 a_{\ell,0}-27 y_0 b_{\ell,0} \qquad \text{and} \qquad b_{\ell+1,0}=x_{0}b_{\ell,0}+y_0 a_{\ell,0}.
\end{equation}
This implies that the spectrum of $\G(3,p^{3t(\ell+1)})$ can be determined by the spectrum of $\G(3,p^{3t\ell})$, recursively.
Thus, the spectrum of $\G(3,p^{3t\ell})$ is determined by the spectrum of $\G(3,p^{3t})$ by induction, as desired.

\medskip
	
Now assume that $s\in \{1,\ldots,t-1\}$ (the case $s=0$ was treated before), and let $a_{0,s},b_{0,s}\in \mathbb{Z}$ with $a_{0,s}\equiv 1\pmod{3}$ and $(a_{0,s},p)=1$ 
such that 
$$4p^s = a_{0,s}^2+27b_{0,s}^2 = \| z_{a_{0,s},b_{0,s}} \|^2$$
with $z_{a_{0,s},b_{0,s}} = a_{0,s} + 3\sqrt{3} i b_{0,s}$. Hence, we have that
$$4p^{t\ell+s} =\|z_{a_{0,s},b_{0,s}} \|^2 \|z_{x_{\ell-1}, y_{\ell-1}} \|^2 = \|z_{a_{0,s},b_{0,s}}  z_{x_{\ell-1},y_{\ell-1}} \|^2 = \|z_{a_{\ell,s},b_{\ell,s}} \|^2$$
where $\{(a_{\ell,s},b_{\ell,s})\}_{\ell\in\mathbb{N}_0}$ also satisfies the recursions
\begin{equation}\label{aies}
	a_{\ell,s}=a_{0,s} x_{\ell-1}-27 b_{0,s} y_{\ell-1} \qquad \text{and}\qquad b_{\ell,s}=a_{0,s} y_{\ell-1}+b_{0,s} x_{\ell-1},
\end{equation}
with $x_{\ell},y_{\ell}$ recursively defined as in \eqref{recursion t1}.

\medskip

\noindent
\textit{Claim 2:} $a_{\ell,s}\equiv 1\pmod{3}$ and $(a_{\ell,s},p)=1$ for all $\ell\in \mathbb{N}_0$.

\smallskip

\noindent
\textit{Proof of the claim:}
For simplicity, here we put $a_{\ell}$ and $b_{\ell}$ instead of $a_{\ell,s}$ and $b_{\ell,s}$, respectively.
Clearly $a_{\ell}\equiv 1\pmod{3}$. On the other hand, 
if $(a_{\ell},p)>1$ then we have that $p\mid b_{\ell}$, 
so there are two cases: $v_{p}(a_\ell)=v_{p}(b_{\ell})$ or $v_{p}(a_\ell)\neq v_{p}(b_{\ell})$. 
Suppose first that $v_{p}(a_\ell)\neq v_{p}(b_{\ell})$ for some $\ell\geq 1$. 
Thus, since the numbers $a_0,b_0,x_{\ell-1}$ are mutually coprime with $p$ and $p>3$, 
we have that $v_p(a_\ell)=v_p(a_0 a_\ell) \ne v_p(27 b_0 b_\ell) = v_p(b_\ell)$ and similarly $v_p(a_\ell)=v_p( x_{\ell-1} a_{\ell}) \ne v_p(27 y_{\ell-1} b_{\ell})=v_p(b_\ell)$.
In this way we get 
\begin{align*}
& \min \{ v_p(a_\ell), v_p(b_\ell) \} = v_p( a_0 a_\ell + 27 b_0 b_\ell) = v_{p}(4 x_{\ell-1}p^{s}) = s, \\[1mm]
& \min \{ v_p(a_\ell), v_p(b_\ell) \} = v_p( x_{\ell-1} a_{\ell}+27 y_{\ell-1} b_{\ell}) = v_{p}(a_{0}p^{\ell t})=\ell t.
\end{align*}
Thus $s=\ell t$, which is absurd 
since $s<t$. 
Hence, we must have that $v_{p}(a_{\ell})=v_{p}(b_\ell)$ for all $\ell\ge 1$.

On the other hand, from the recursions \eqref{recursion t1} and \eqref{aies} we have that
\begin{equation}\label{aies 2}
	a_{\ell}=x_0 a_{\ell-1}-27 y_0 b_{\ell-1} \qquad \text{and}\qquad b_{\ell}=x_0 b_{\ell-1}+y_0 a_{\ell-1}.
\end{equation}
Combining both recursions we get 
$$a_{\ell+1}=x_0 a_{\ell}-27 y_0 (x_0 b_{\ell-1}+y_0 a_{\ell-1})=x_0 a_{\ell}-27 y_0^2  a_{\ell-1} -27 x_0 y_0 b_{\ell-1},$$
and, by using that $p^t=x_0^2+27 y_0^2$, we arrive at
\begin{equation}\label{recursion imp a}
a_{\ell+1}=x_0 a_{\ell}-p^{t}a_{\ell-1}+x_0^2 a_{\ell-1} -27 x_0 y_0 b_{\ell-1}=2x_0a_{\ell}-p^t a_{\ell-1}.
\end{equation}
Now, recall that $v_{p}(a_\ell)\le s<t$ for all $\ell\ge 0$ and so $v_{p}(2x_0a_{\ell})\neq v_p(p^t a_{\ell-1})$. Thus, we have that
$$v_{p}(a_{\ell+1}) = \min \{v_{p}(a_{\ell}), t + v_{p}(a_{\ell-1})\}=v_{p}(a_{\ell})$$
for all $\ell\ge 1$. Hence, it is enough to see that $v_{p}(a_1)=0$. 
Let $M=v_{p}(a_1)=v_{p}(b_1)$. First notice that $M<s$. Indeed, if $a_1=p^s \hat{x}$ and $b_1=p^s \hat{y}$, thus 
$$p^{2s}(\hat{x}^2+27 \hat{y}^2)=4p^s$$
which cannot happen since $s>1$. On the other hand, by \eqref{aies 2} we have that
$$a_{1}^2=x_{0}^2 a_{0}^2-54x_0 a_0 y_0 b_0 +27^2 y_0^2 b_0^2=(p^t-27 y_0^2)(4p^s-27 b_0^2)-54x_0 a_0 y_0 b_0 +27^2 y_0^2 b_0^2$$ 
and so we obtain
$$a_{1}^2=4p^{t+s}-4\cdot 27 p^s y_0^2-27p^t b_{0}^2+27^2y_0^2 b_0^2-54x_0 a_0 y_0 b_0 +27^2 y_0^2 b_0^2.$$
Thus, using that $a_1=a_0 x_0-27 b_0 y_0$, we get
$$a_{1}^2=4p^{t+s}-4\cdot 27 p^s y_0^2-27p^t b_{0}^2-54 y_0 b_0 a_1.$$
Since $t+s,t,s,M$ are all different, in fact $M<s<t<t+s$, and $y_0,b_0$ are both coprime with $p$, we obtain that
$$2M=v_p(a_1^2)=\min \{t+s,t,s,M\}=M.$$
Hence $M=0$, as desired. So, we obtain that $(a_{\ell},p)=1$ for all $\ell \ge 1$. 
This proves the claim. \hfill $\diamond$

\smallskip

Therefore, if we put $n_{\ell,s}=\frac{p^{3(t\ell+s)}-1}{3}$ then, by Theorem \ref{gp3q}, the spectrum of $\G(3,p^{3(t\ell+s)})$ is given by
\begin{equation} \label{spec lt+s}
\mathrm{Spec}\,\G(3,p^{3(t\ell+s)}) = \big\{ [n]^1, \big[ \tfrac{a p^{t\ell+s}-1}{3}\big]^{n}, \big[ \tfrac{-\frac{1}{2} (a+9b)p^{t\ell+s}-1}{3}\big]^{n}, 
\big[ \tfrac{-\frac 12 (a-9b) p^{t\ell+s}-1}{3}\big]^{n} \big\}.
\end{equation}
where we write $n=n_{\ell,s}$, $a=a_{\ell,s}$ and $b=b_{\ell,s}$ for simplicity.

In order to prove that the spectrum of $\G(3,p^{3(t\ell+s)})$ is determined by the spectra of $\G(3,p^{3s})$ and $\G(3,p^{3t})$, 
it is enough to put $a_{\ell,s}$ and $b_{\ell,s}$ in terms of $a_{0,s},b_{0,s}$ and $x_0,y_0$.
Notice that the sequence of $b_{\ell,s}$'s satisfy the same recurrence as in \eqref{recursion imp a}, i.e.\@ 
$$b_{\ell+1,s}=2x_0 b_{\ell,s}-p^t b_{\ell-1,s}.$$
By solving this two terms linear recurrence and recalling that $a_{1,s}=a_{0,s}x_0-27 b_{0,s}y_0$ and $b_{1,s}=a_{0,s}y_0+x_0b_{0,s}$, we obtain that
\begin{equation} 
\begin{aligned} \label{rec aell}
& a_{\ell,s} = \tfrac 12 (a_{0,s}+3\sqrt{3}b_{0,s} i) (x_0+3\sqrt{3}y_0 i)^{\ell} + 
\tfrac{1}{2}(a_{0,s}-3\sqrt{3}b_{0,s} i)(x_0-3\sqrt{3}y_0 i)^{\ell}, \\
& b_{\ell,s}= \tfrac{1}{2}(b_{0,s}-\tfrac{\sqrt{3}}{9}a_{0,s} i) (x_0+3\sqrt{3}y_0 i)^{\ell}+\tfrac{1}{2}(b_{0,s}+\tfrac{\sqrt{3}}{9}a_{0,s} i)(x_0-3\sqrt{3}y_0 i)^{\ell}.
\end{aligned}
\end{equation}
In this way, for every $\ell \in \N$, $a_{\ell,s}$ and $b_{\ell,s}$ can be put in terms of $a_{0,s},b_{0,s}$ and $x_0,y_0$ only, as we wanted to show.
\end{proof}

\begin{rem} \label{rem rec}
As in the case $s>0$, from \eqref{recursions} for $s=0$ 
we have that both $\{a_{\ell,0}\}$ and $\{b_{\ell,0}\}$ satisfy the recursion
$$c_{\ell+2,0}=2x_0 c_{\ell+1,0}-p^t c_{\ell,0}$$
for any $\ell>0$. 
Finally, since $a_{2,0}=a_{1,0}x_0-27 b_{1,0}y_0$ and $b_{2,0}=a_{1,0}y_0+x_0b_{1,0}$, we obtain that \eqref{rec aell} also holds for $s=0$ in the following way
\begin{equation*}
\begin{aligned} 
	& a_{\ell,0} = \tfrac 12 (a_{1,0}+3\sqrt{3}b_{1,0} i) (x_0+3\sqrt{3}y_0 i)^{\ell-1} + 
	\tfrac{1}{2}(a_{1,0}-3\sqrt{3}b_{1,0} i)(x_0-3\sqrt{3}y_0 i)^{\ell-1}, \\
	& b_{\ell,0}= \tfrac{1}{2}(b_{1,0}-\tfrac{\sqrt{3}}{9}a_{1,0} i) (x_0+3\sqrt{3}y_0 i)^{\ell-1}+\tfrac{1}{2}(b_{1,0}+\tfrac{\sqrt{3}}{9}a_{1,0} i)(x_0-3\sqrt{3}y_0 i)^{\ell-1}.
\end{aligned}
\end{equation*}
Moreover, by taking into account that $a_{1,0}=-2x_0$ and $b_{1,0}=-2y_0$  we obtain that
\begin{equation}
	\begin{aligned} \label{rec aell0} 
		& a_{\ell,0} = - (x_0+3\sqrt{3}y_0 i)^{\ell} -(x_0-3\sqrt{3}y_0 i)^{\ell}, \\
		& b_{\ell,0} = -(y_0-\tfrac{\sqrt{3}}{9}x_0 i) (x_0+3\sqrt{3}y_0 i)^{\ell-1} - 
		(y_0+\tfrac{\sqrt{3}}{9}x_0 i)(x_0-3\sqrt{3}y_0 i)^{\ell-1},
	\end{aligned}
\end{equation}
where $x_0$ and $y_0$ are the solutions of $p^t=X^2+27 Y^2$ with $(x_0,p)=1$ and $x_0\equiv 1 \pmod{3}$.
\end{rem}

\smallskip

Notice that \eqref{rec aell} can be written in the notation of \eqref{zxy} as 
\begin{equation} \label{albl}
\begin{aligned} 
a_{\ell,s} = \tfrac 12 (z_{a,b} z_{x,y}^\ell + \bar z_{a,b} \bar z_{x,y}^\ell) = Re(z_{a,b} z_{x,y}^\ell), \\
b_{\ell,s} = \tfrac 12 (z_{b, \tilde a} z_{x,y}^\ell + \bar z_{b,\tilde a} \bar z_{x,y}^\ell) = Re(z_{b, \tilde a} z_{x,y}^\ell), 
\end{aligned}
\end{equation}
where $(a,b,x,y)=(a_{0,s},b_{0,s},x_0,y_0)$ and $\tilde a = \tfrac{1}{27} a$. 
Similarly, expressions \eqref{rec aell0} can be written as the real part of a complex number. 

\smallskip 

Theorem \ref{gp3q} provides the spectrum of $\Gamma(3,q)$ explicitly, with $q=p^{3r}$ and $p$ a prime of the form $3h+1$, in terms of an integer solution $(a,b)$ of the equation $4q =X^2 +27Y^2$ with $a\equiv 1\pmod 3$ and $(a,p)=1$. 
If one wants the spectrum of $\Gamma(3,q^t)$, one needs to obtain an integer solution $(a_t,b_t)$ of the equation 
$4q^t =X^2 +27Y^2$ with $a_t \equiv 1 \pmod 3$ and $(a_t,p)=1$. This may be tedious and operationally costly.
However, Theorem \ref{spectra der 3 gen} ensures that we can obtain all the spectra of the GP-graphs $\Gamma(3,q^{t\ell})$ %($s=0$) 
from a unique base solution of the first equation $q^t =X^2 +27Y^2$. Moreover, the spectrum of $\Gamma(3,p^{t\ell})$ is given explicitly in the previous proof by expressions \eqref{spec lt+s} and \eqref{rec aell}.

\begin{exam} \label{ex 3.3}
Let $p=7$. We look for the minimal $t$ such that $7^t=X^2+27Y^2$ has an integer solution $(x,y)$ with $x\equiv 1\pmod 3$ and $(x,7)=1$. In this case $t=3$ with solution $(x_0,y_0)=(10,3)$, since $7^3=10^2 +27 \cdot 3^2$ 
(for $t=1$ there are no integer solutions and for $t=2$ we have the trivial solution $(7,0)$ but it is not of the required form). Consider $s= 1$. By \eqref{spec lt+s}, the spectrum of $\Gamma(3,7^{9\ell+3})$ for any $\ell \in \N$ is given by 
	$$\mathrm{Spec}\,\G(3,7^{9\ell+3}) = \big\{ [n_{\ell}]^1, \big[ \tfrac{a_{\ell}7^{3\ell+1}-1}{3}\big]^{n_{\ell}}, \big[ \tfrac{-\frac{1}{2} (a_{\ell}+9b_{\ell}) 7^{3\ell+1}-1}{3}\big]^{n_{\ell}}, 
	\big[ \tfrac{-\frac 12 (a_{\ell}-9b_{\ell}) 7^{3\ell+1}-1}{3}\big]^{n_{\ell}} \big\}$$
where $n_{\ell} = \frac{7^{9\ell+3}-1}{3}$
and the numbers $a_\ell$ and $b_\ell$ are, by \eqref{rec aell}, as follows
\begin{align*}
& a_{\ell} = \tfrac{1}{2}( 1- 3\sqrt{3}i)(10 - 9\sqrt{3}i)^{\ell} + \tfrac{1}{2}(1 + 3\sqrt{3}i)(10 + 9 \sqrt{3}i)^{\ell}, \\
& b_{\ell}= \tfrac{1}{2}(1 + \tfrac{\sqrt{3}i}{9})(10 - 9\sqrt{3}i)^{\ell} + \tfrac{1}{2}(1 - \tfrac{\sqrt{3}i}{9})(10 + 9\sqrt{3}i)^{\ell},
\end{align*}
where $a_0= 1$ and $b_0=1$ since $4\cdot 7=1^2+27 \cdot 1^2$.

In Table 1 we give the spectrum of $\G(3,7^{9\ell +3})$ for the first five values of $\ell$. For simplicity we only list the non principal eigenvalues $\{ \lambda_1, \lambda_2, \lambda_3 \}$ without the multiplicities and separately the principal eigenvalues.
	\begin{table}[h!]
		\caption{Spectrum of $\G(3,7^{9\ell +3})$.}
		\begin{tabular}{|c|c|c|c|c|}
			\hline
			$\ell$ & $a_{\ell}$ & $b_{\ell}$ & $q$ & Non-principal eigenvalues of $\G(3,7^{9\ell +3})$\\ \hline 
			0 & $1$ & $1$ & $7^3$ & $\{9, 2, -12\}$\\
			1 & $-71$ &$13$ & $7^{12}$ & $\{ 75231, -18408, -56824 \}$\\
			2 & $-1763$ & $-83$ & $7^{21}$ & $\{ 344515488, 139453281, -483968770 \}$\\ 
			3 & $-10907$ & $-6119$ & $7^{30}$ & $\{3106191996420, -1026985846948, -2079206149473\}$\\
			4 & $386569$ & $-93911$ & $7^{39}$ & $\{ 12484762621341194, 7406034473827068, -19890797095168263\}$\\
			\hline 
		\end{tabular}
	\end{table}

\noindent 
The principal eigenvalues %(regularity degree) 
are given by $n_0 = 114$, $n_1 = 4613762400$, $n_2 = 186181954694428002$, $n_3 = 7513113430230752695954416$ and 
$n_4 = 303181226709953713606735006629714$.
Note that we do not have to solve $4\cdot 7^{3\ell+1}=X^2+27Y^2$ for each value of $\ell \ge 1$. 
\hfill $\lozenge$
\end{exam}

Recall that an integer $a$ is a \textit{cubic residue} modulo a prime $p$ if $a\equiv x^3 \pmod p$ for some integer $x$. 
By Euler's criterion, $a$ is a cubic residue mod $p$, with $(a,p)=1$, if and only if 
\begin{equation} \label{ap cond}
a^{\frac{p-1}{d}}\equiv 1 \pmod p
\end{equation}
where $d=(3,p-1)$. 
We have the following direct consequence of Theorem \ref{spectra der 3 gen}.
\begin{thm} \label{res cubic}
Let $p$ be a prime with $p\equiv 1 \pmod{3}$.
If $2$ is a cubic residue modulo $p$, then the spectrum of $\G(3,p^3)$ determines 
the spectrum of $\G(3,p^{3\ell})$ for every $\ell \in \N$. In this case, the spectrum of $\G(3,p^{3\ell})$ is given by
\begin{equation} \label{spec p3l}
\mathrm{Spec}\,\G(3,p^{3\ell}) = \big\{ [n_\ell]^1, \big[ \tfrac{a_\ell p^{\ell}-1}{3}\big]^{n_\ell}, \big[ \tfrac{-\frac{1}{2} (a_\ell+9b_\ell)p^{\ell}-1}{3}\big]^{n_\ell}, \big[ \tfrac{-\frac 12 (a_\ell-9b_\ell) p^{\ell}-1}{3}\big]^{n_\ell} \big\}.
\end{equation}
with $n_\ell = \frac{p^{3\ell}-1}{3}$ and where $a_\ell, b_\ell$ are the numbers $a_{\ell,0}, b_{\ell,0}$ given in \eqref{rec aell0}
\begin{equation*}
\begin{aligned} 
& a_{\ell,0} = - (x_0+3\sqrt{3}y_0 i)^{\ell} -(x_0-3\sqrt{3}y_0 i)^{\ell}, \\
& b_{\ell,0} = -(y_0-\tfrac{\sqrt{3}}{9}x_0 i) (x_0+3\sqrt{3}y_0 i)^{\ell-1} - 
(y_0+\tfrac{\sqrt{3}}{9}x_0 i)(x_0-3\sqrt{3}y_0 i)^{\ell-1},
\end{aligned}
\end{equation*}
where $x_0$ and $y_0$ are the solutions of $p=X^2+27 Y^2$ with $(x_0,p)=1$ and $x_0\equiv 1 \pmod{3}$.
\end{thm}

\begin{proof}
A classic result in number theory, conjectured by Euler and first proved by Gauss using cubic reciprocity, asserts that 
(see for instance \cite{C})
\begin{equation} \label{2res conds}
	p=x^2+27 y^2 \quad \text{for some $x,y\in\mathbb{Z}$} \qquad \Leftrightarrow \qquad 
		\begin{cases}
			p\equiv 1\pmod{3} \quad \text{ and,} \\ 2 \text{ is a cubic residue modulo $p$}.
		\end{cases}
\end{equation}
By hypothesis we have that $p\equiv 1\pmod{3}$ and $2$ is a cubic residue modulo $p$, so there exist 
	$x,y \in\mathbb{Z}$ such that $p=x^2+27 y^2$. Moreover, since either $x$ or $-x$ is congruent to $1$ mod $p$, we choose the solution $(z,y)$, where $z\in \{\pm x\}$ with $z\equiv 1 \pmod 3$.
	Thus, the assertion follows directly from Theorem \ref{spectra der 3 gen} with $t=1$ and $s=0$.
\end{proof}

Note that by \eqref{ap cond}, if $p \equiv 1 \pmod 3$ as in the theorem, $2$ is a cubic residue modulo $p$ if and only if $2^{\frac{p-1}3}  \equiv 1 \pmod p$. That is, if $p=3k+1$ then $2^k \equiv 1 \pmod {3k+1}$.
Thus, the first primes $p$ of the form $3k+1$ for which $2$ is a cubic residue are $31$ and $43$ since 
$2^{10} \equiv 1 \pmod{31}$ and $2^{14} \equiv 1 \pmod{43}$. In fact, $4^3 \equiv 1 \pmod{31}$ and $20^3 \equiv 1 \pmod{43}$.

\begin{exam}  
Let $p=31$. We know that $2$ is a cubic residue modulo $31$ and in this case we have 
$31= 2^2 +27\cdot 1^2$. %In the notation of Theorem \ref{spectra der 3 gen}, 
We take the solutions $x_0=-2$ and $y_0=1$ of $31=X^2+37Y^2$.
By Theorem \ref{res cubic} and \eqref{rec aell0} in Remark \ref{rem rec} for every $\ell >0$ we have that
the spectrum of $\G(4,31^3)$ determines the spectrum of $\G(4,31^{3\ell})$ for every $\ell$ and, from \eqref{spec p3l}, 
it is given by
\begin{equation*} 
\mathrm{Spec}\,\G(3,31^{3\ell}) = \big\{ [n_\ell]^1, \big[ \tfrac{a_\ell 31^{\ell}-1}{3}\big]^{n_\ell}, \big[ \tfrac{-\frac{1}{2} (a_\ell+9b_\ell)31^{\ell}-1}{3}\big]^{n_\ell}, \big[ \tfrac{-\frac 12 (a_\ell-9b_\ell) 31^{\ell}-1}{3}\big]^{n_\ell} \big\}.
\end{equation*}
with $n_\ell = \frac{31^{3\ell}-1}{3}$ and where $a_\ell, b_\ell$ are the numbers $a_{\ell,0}, b_{\ell,0}$ given in \eqref{rec aell0}, that is
\begin{align*} 
& a_{\ell} = - (-2+3\sqrt{3} i)^{\ell} -(-2-3\sqrt{3} i)^{\ell}, \\
& b_{\ell} = -(1+\tfrac{2\sqrt{3}}{9} i) (-2+3\sqrt{3} i)^{\ell-1} - 
(1-\tfrac{2\sqrt{3}}{9} i)(-2-3\sqrt{3} i)^{\ell-1}.
\end{align*}
In Table 2 we give the spectrum of $\G(3,31^{\ell})$ for the first five values of $\ell$ (we follow the same notation as in Example \ref{ex 3.3})
\begin{table}[h!]
	\caption{First values of $a_{\ell}$ and $b_{\ell}$ for $\G(3,31^{3\ell})$}
	\begin{tabular}{|c|c|c|c|c|}
		\hline
		$\ell$ & $a_{\ell}$ & $b_{\ell}$ & $q$ & Non-principal eigenvalues of $\G(3,31^{3\ell})$\\ \hline 
		1 & $4$ & $-2$ & $31^3$ & $\{  72, 41, -114\}$\\
		2 & $46$ &$8$ & $31^6$ & $\{  14735, 4164, -18900, \}$\\
		3 & $-308$ & $30$ & $31^9$ & $\{ 2869866, 188676, -3058543 \}$\\ 
		4 & $-194$ & $-368$ & $31^{12}$ & $\{539644104, -59721025, -479923080\}$\\
		5 & $10324$ & $542$ & $31^{15}$ & $\{ 98522451641, -25985726058, -72536725584\}$\\
		\hline 
	\end{tabular}
\end{table}

\noindent where the principal eigenvalues are
$n_0 = 114$, $n_1 = 4613762400$, $n_2 = 186181954694428002$, $n_3 = 7513113430230752695954416$ and $n_4 = 303181226709953713606735006629714$.
\hfill $\lozenge$
\end{exam}

The following result for GP-graphs $\G(4,q)$ is in the same vein of Theorem \ref{spectra der 3 gen} for GP-graphs $\G(3,q)$. However, these results are somehow different, since in this case we do not need the assumption of an initial solution of an equation (as in Theorem \ref{spectra der 3 gen}), although it covers different cases.

\begin{thm} \label{spectra der 4}
	If $p$ is a prime with $p\equiv 1 \pmod{4}$, then the spectrum of $\G(4,p^{4\ell})$ is determined by the spectrum of $\G(4,p^4)$ for every $\ell \in \N$. Moreover, the spectrum of $\G(4,p^{4\ell})$ is given by
	$$Spec(\G(4,p^{4\ell})) = \big\{ [n_{\ell}]^1, \big[\tfrac{p^{2\ell} + 4d_{\ell}p^{\ell}-1}{4}\big]^{n_{\ell}}, \big[\tfrac{p^{2\ell} - 4d_{\ell}p^{\ell}-1}{4}\big]^{n_{\ell}}, 
	\big[\tfrac{-p^{2\ell} + 2c_{\ell} p^{\ell}-1}{4}\big]^{n_{\ell}}, \big[\tfrac{-p^{2\ell} - 2c_{\ell} p^{\ell}-1}{4}\big]^{n_{\ell}}\big\},$$
	where $n_\ell=\frac{p^{4\ell}-1}4$ and the numbers $c_\ell$ and $d_\ell$ are given by
	\begin{equation} \label{rec cl,dl}
c_{\ell}= \tfrac 12 (c_1+2 d_1 i)^{\ell}+ \tfrac 12 (c_1-2 d_1 i)^{\ell} \qquad \text{and} \qquad
d_{\ell}= -\tfrac{i}{4} (c_1+2 d_1 i)^{\ell}+ \tfrac{i}{4}(c_1-2 d_1 i)^{\ell},
	\end{equation}
where $c_1$ and $d_1$ are integral solutions of $p^2=X^2+ 4Y^2$ with $c_1 \equiv 1 \pmod 4$ and $(c_1,p)=1$.
\end{thm}

\begin{proof}
	It is well known that the equation $p^2=X^2+ 4Y^2$ with $p\equiv 1\pmod{4}$ always has a solution $(x,y)$ satisfying $(x,p)=1$   (see Remark \ref{solutions eq}).
	Let $c_1, d_1$ be the solution of the above equation with $c_1\equiv 1\pmod{4}$ and $(c_1,p)=1$. 
	Notice that if we take 
	$z_{x,y}=x+2iy$, then $p^2 = \|z_{c_1,d_1} \|^2$, so we have that 
	$$p^{2\ell} = \|z_{c_1,d_1}^{\ell}\|^2.$$
	As in the proof of Theorem \ref{spectra der 3 gen}, %In the same way as before, 
	we can put 
	$z_{c_{1},d_{1}}^{\ell}=: z_{c_{\ell},d_{\ell}}$, where $c_{\ell},d_{\ell}$ are defined recursively as follows
	\begin{equation}\label{recursion t4}
		c_{\ell+1}=c_1 c_{\ell}-4d_{1} d_{\ell} \qquad \text{and} \qquad d_{\ell+1}= c_1 d_{\ell}+d_1 c_{\ell}.
	\end{equation}	
	Both sequences $\{c_{\ell}\}_{\ell \in \N_0}$ and $\{d_{\ell}\}_{\ell \in \N_0}$ also satisfy the recursion 
	\begin{equation}\label{recursion cd4}
		r_{\ell+1}=2c_1 r_{\ell}- p^2 r_{\ell-1}.
	\end{equation}		
	Proceeding similarly as in the proof of Theorem \ref{spectra der 3 gen}, 
	we can show that
	$$d_{\ell+1}=d_1 \Big(\sum_{i=1}^{\ell} c_{i} c_1^{\ell-i} + c_{1}^{\ell} \Big) \quad \text{and} \quad  
	c_{\ell} = \tfrac{1}{d_1}(d_{\ell+1} - c_1 d_{\ell}) = \sum_{i=1}^{\ell} c_{i} c_1^{\ell-i} -c_1 \sum_{i=1}^{\ell-1}c_{i} c_1^{\ell-1-i} $$
	so we have that
	\begin{equation}
		c_{\ell+1}= c_1 \sum_{i=1}^{\ell} c_{i} c_1^{\ell-i}  - p^2 \sum_{i=1}^{\ell-1}c_{i} c_1^{\ell-1-i}- 4d_{1}^2 c_{1}^{\ell-1}.
	\end{equation} 
	As in proof of Theorem \ref{spectra der 3 gen}, we can show that $(c_{\ell},p)=1$ and $c_{\ell}\equiv 1\pmod{4}$. 
	Hence, by Theorem \ref{gp4q} we have that the spectrum of $\G(4,p^{4\ell})$ is given as in the statement.
	
	In order to prove that the spectrum of $\G(4,p^{4\ell})$ is determined by the spectrum of $\G(4,p^{4})$, 
	it is enough to put every $c_{\ell}$ and $d_{\ell}$ in terms of $c_1$ and $d_1$ only.
	By solving the linear recurrence \eqref{recursion cd4} and by recalling that $c_2=c_{1}^2-4 d_1^2$ and $d_2=2c_1 d_1$, we obtain that  $c_\ell$ and $d_\ell$ are as given in \eqref{rec cl,dl}.
	Therefore, the spectrum of $\G(4,p^{4\ell})$ is determined by the spectrum of $\G(4,p^4)$, as desired.	
\end{proof}

\begin{exam}
	Let $p=5$. Since $5^2=3^2+4\cdot 2^2$, we take $c_1=-3$ and $d_1=2$.
	The spectrum of $\Gamma(4,5^{4\ell})$ is given for any $\ell \in \N$ by 
	$$Spec(\G(4,5^{4\ell})) = \big\{ [n_{\ell}]^1, \big[\tfrac{5^{2\ell} + 4d_{\ell}5^{\ell}-1}{4}\big]^{n_{\ell}}, \big[\tfrac{5^{2\ell} - 4d_{\ell}p^{\ell}-1}{4}\big]^{n_{\ell}}, 
	\big[\tfrac{-5^{2\ell} + 2c_{\ell} 5^{\ell}-1}{4}\big]^{n_{\ell}}, \big[\tfrac{-5^{2\ell} - 2c_{\ell} 5^{\ell}-1}{4}\big]^{n_{\ell}}\big\}$$
	with $n_\ell=\frac{5^{4\ell}-1}4$ and where, by \eqref{rec cl,dl}, 
	$$c_{\ell}= \tfrac 12 (-3+4 i)^{\ell}+ \tfrac 12 (-3-4 i)^{\ell} \qquad \text{and } \qquad 	d_{\ell}= -\tfrac{i}{4} (-3+4 i)^{\ell}+ \tfrac{i}{4}(-3-4i)^{\ell}.$$
	In Table 3 we give the spectrum of $\G(4,5^{4\ell})$ for the first five values of $\ell$ (we follow the same notation as in Example \ref{ex 3.3})	
	\begin{table}[h!]
		\caption{Spectrum of $\G(4,5^{4\ell})$}
		\begin{tabular}{|c|c|c|c|}
			\hline
			$\ell$ & $c_{\ell}$ & $d_{\ell}$ & Non-principal eigenvalues of $\G(4,5^{4\ell})$\\ \hline 
			1 & $-3$ & $2$ & $\{ 16, -4, -14, 1\}$\\
			2 & $-7$ & $-12$ & $\{  -144, 456, -244, -69\}$\\
			3 & $117$ & $22$ & $\{ 6656, 1156, 3406, -11219\}$\\ 
			4 & $-527$ & $168$ & $\{202656, -7344, -262344, 67031\}$\\
			5 & $237$ & $-1558$ & $\{ -2427344, 7310156, -2071094, -2811719\}$\\
			\hline 
		\end{tabular}
	\end{table}

\noindent where 
$n_1 = 156$, $n_2 = 97656$, $n_3 = 61035156$, $n_4 = 38146972656$ and $n_5 = 23841857910156$ 
are the principal eigenvalues.
\hfill $\lozenge$
\end{exam}

\section{Energy}
In this section we first study the energy of the graphs $\G=\G^*(k,q)$ for $k=3,4$ where $q=p^r$ with $p$ prime. Then, we give conditions on $\Gamma$ ensuring that $\G$ and $\bar \G$ are complementary equienergetic (i.e.\@ they have the same energy), without computing the energies of $\Gamma$ and $\bar \G$. It turns out that we will only need to know the sign of the eigenvalues of $\G$.

We begin by studying the energies of $\G^*(3,q)$ and $\G^*(4,q)$. We will give the exact values in the semiprimitive case $p\equiv -1 \pmod k$ with $k=3,4$. However, in the non-semiprimitive case $p\equiv 1 \pmod k$ we can only give bounds.

\begin{prop} \label{energies}
For the energies of $\G^*(3,q)$ and $\G^*(4,q)$ we have the following:

\noindent $(a)$ 
If $p \equiv 1 \pmod k$ with $k=3$ or $k=4$ then
\begin{gather*}
n(1+\tfrac 13 |2a \sqrt[3]{q}+1|) \le E(\G^*(3,q)) \le n \big( 1 + \tfrac 23 (|a| \sqrt[3]{q}+1) +3|b|), \\[.5em] 
n(\sqrt{q}+1) \le E(\G^*(4,q)) \le n \big( \sqrt{q}+1 +(|c|+2|d|)\sqrt[4]{q} \big).
%\end{split}
\end{gather*}

\noindent $(b)$ If $p\equiv -1 \pmod k$ then we have: 

$(i)$ If $p \equiv 2 \pmod 3$ then 
$$E(\G^*(3,q)) = \begin{cases}
2n \frac{2\sqrt{q}+1}{3} & \qquad \text{if } m \equiv 0 \pmod 4, \\[.4em]
4n \frac{\sqrt{q}+1}{3} & \qquad \text{if } m \equiv 2 \pmod 4.
\end{cases}$$

$(ii)$ If $p \equiv 3 \pmod 4$ then 
$$E(\G^*(4,q)) = \begin{cases}
n \frac{3\sqrt{q}+1}{2} & \qquad \text{if } m \equiv 0 \pmod 4, \\[.4em]
3n \frac{\sqrt{q}+1}{2} & \qquad \text{if } m \equiv 2 \pmod 4.
\end{cases}$$
\end{prop}

\begin{proof}
It is clear from ($iii)$ in Remark \ref{equinoiso} that the graphs $\G(k,q)$ and $\G^+(k,q)$ are equienergetic for $k=3,4$. So, it is enough to compute the energy for the graphs $\G(3,q)$ and $\G(4,q)$. The result follows by the definition of energy in \eqref{energy} and by applying Theorems \ref{gp3q} and \ref{gp4q}. The equalities in parts $(i)$ and $(ii)$ in ($b$) of the statement are straightforward consequences of part ($b$) of the aforementioned theorems. The bounds in part ($a$) of the statement are deduced from part ($a$) of these theorems, by applying the triangle inequality of real numbers. 
\end{proof}

\subsubsection*{Equienergy}
We now study when the graphs $\G(3,q)$ and $\G(4,q)$ are equienergetic to their corresponding complements $\bar \G(3,q)$ and 
$\bar \G(4,q)$. 
In the case that $\G(k,q)$ is semiprimitive with $k=3,4,$ the answer is already known. In fact, we have that $q=p^{2t}$ with $p$ prime satisfying $p \equiv -1 \pmod k$ and, from Proposition 6.6 in \cite{PV3}, we obtain that 
$\{\G(3,p^{2t}), \bar \G(3,p^{2t})\}$ and $\{\G(4,p^{2t}), \bar \G(4,p^{2t})\}$ 
are pairs of equienergetic non-isospectral complementary graphs if and only if $t$ is odd. That is,  
$$\{\G(3,p^{4s+2}), \bar \G(3,p^{4s+2})\} \qquad \text{and} \qquad \{\G(4,p^{4s+2}), \bar \G(4,p^{4s+2})\}$$ 
are pairs of equienergetic non-isospectral complementary graphs for any $s\in \N$.

The following result gives a condition for $\G(k,q)$ such that $\G(k,q)$ and $\bar \G(k,q)$ are equienergetic graphs in the non-semiprimitive case.
We recall that the principal eigenvalue of a regular graph is its degree of regularity. 

\begin{thm} \label{equien Gp comp}
Let $q=p^m$ with $p$ prime, $m \in \N$ and $k=3,4$ such that $k\mid \frac{q-1}{p-1}$ and $(k,q)$ is not a semiprimitive pair. Then, the graphs $\G(k,q)$ and $\bar \G(k,q)$ are equienergetic if and only if among the non principal eigenvalues of $\G(k,q)$ exactly one is positive. In this case, the graphs $\{\G(k,q), \G^+(k,q), \bar{\G}(k,q) \}$ are mutually equienergetic and non-isospectral.
\end{thm} 

\begin{proof}
We begin by showing that $\G(k,q)$ and $\bar{\G}(k,q)$ are equienergetic graphs if and only if $\G(k,q)$ has only one positive non-principal eigenvalue for $k=3,4$.

Let us first consider $k=3$. Since $(3,q)$ is not semiprimitive, then  $p\equiv 1\pmod 3$ and $m=3t$ for some $t\in\mathbb{N}$.
Moreover, by ($a$) in Theorem \ref{gp3q}, the spectrum of $\G(3,q)$ is
	$$\mathrm{Spec}\,\G(3,q)=\big\{ [n]^1, \big[ \tfrac{a\sqrt[3]{q}-1}{3}\big]^{n}, \big[ \tfrac{-\frac{1}{2} (a+9b)\sqrt[3]{q}-1}{3}\big]^{n}, 
	\big[ \tfrac{-\frac 12 (a-9b) \sqrt[3]{q}-1}{3}\big]^{n} \big\}$$
where $n=\frac{q-1}3$ and the integers $a,b$ satisfy $4\sqrt[3]{q}=a^{2}+27b^2$ with $a\equiv 1\pmod{3}$ and $(a,p)=1$. Hence, if we denote by $A_3=\{\lambda \in \mathrm{Spec}(\G) : \lambda \ne n \}$ the set of non-principal eigenvalues of $\G(3,q)$, then its energy 
can be written as 
	\begin{equation} \label{En G3}
	E(\G(3,q))= n \big( 1 + \sum_{\lambda\in A_3}|\lambda| \big). 
	\end{equation}

Now, recall that the eigenvalues of $\bar \G(3,q)$ are $q-1-n=(k-1)n=2n$ and $-1-\lambda$ for $\lambda$ a non-principal eigenvalue of $\G(3,q)$.
Hence, we have 
$$E(\bar\G (3,q))= 2n + n \sum_{\lambda\in A_3} |\lambda+1|.$$
Since every element in $A_3$ is a non-zero integer we have
	\begin{equation*} \label{mod Lambda}
	|\lambda+1| = |\lambda| +\mathrm{sign}(\lambda) 
	\end{equation*} 
for all $\lambda \in A_3$, where sign$(t)$ denotes the sign of an integer $t$ (i.e.\@ sign$(t) =1$ if $t>0$ and sign$(t)=-1$ if $t<0$).
Thus, the energy of $\bar \G(3,q)$ takes the following form
	\begin{equation} \label{En G3bar}
	E(\bar\G (3,q))= n \big(2+ \sum_{\lambda\in A_3} |\lambda| + \sum_{\lambda\in A_3} \mathrm{sign}(\lambda) \big).
	\end{equation}
Therefore, by \eqref{En G3} and \eqref{En G3bar}, we have that $E(\G(3,q))=E(\bar \G(3,q))$ if and only if 
	$$\sum_{\lambda\in A_3}\mathrm{sign}(\lambda) = -1.$$
This can only happen if $\G(3,q)$ has only one non-principal positive eigenvalue and the other two non-principal eigenvalues are negative. In fact, $\#A_3=3$ since $b$ in \eqref{ab27} cannot be $0$ (if $b=0$ in \eqref{ab27}, then $(a,p)>1$, which is absurd).
	
The case $k=4$ is similar. If we denote by $A_4$ the set of non-principal eigenvalues of $\G(4,q)$, proceeding as before (we omit the details) we get that 
$E(\G(4,q))=E(\bar \G(4,q))$ if and only if 
	$$\sum_{\lambda\in A_4}\mathrm{sign}(\lambda) = -2.$$
This can only happen if among the non-principal eigenvalues there is exactly one which is positive. This is because
	$\#A_4=4$ since $d\ne 0$ in \eqref{c4d}.

Now, assume we are in the case that the graphs $\G(k,q)$ and $\bar{\G}(k,q)$ are equienergetic. They are also non-isospectral since they have different degrees of regularity. 
Notice that $\G(k,q)$ and $\G^+(k,q)$ are equienergetic non-isospectral graphs by $(iii)$ in Remark \ref{equinoiso}. By transitivity, the three graphs $\G(k,q)$, $\G^+(k,q)$ and $\bar{\G}(k,q)$ are equienergetic. Finally, $\bar \G$ and $\G^+$ are non-isospectral since they have different degrees of regularity. This completes the proof.
\end{proof}

\begin{coro} \label{cond equien}
Let $q=p^m$ with $p$ prime, $m \in \N$, and let $k=3,4$ such that $k\mid \frac{q-1}{p-1}$ and $(k,q)$ is not a semiprimitive pair.
Suppose that $a,b$ and $c,d$ are pairs of integers satisfying conditions \eqref{ab27} and \eqref{c4d} respectively.

\begin{enumerate}[$(i)$]	
\item If $a>9|b|$ or else if $a<0$ and $-a<9|b|$ then $\{\G(3,q), \G^+(3,q), \bar{\G}(3,q) \}$ are equienergetic and non-isospectral.

\item If $2|c|<\sqrt[4]{q}<4|d|$ then the graphs $\{\G(4,q), \G^+(4,q), \bar{\G}(4,q) \}$ are equienergetic and non-isospectral. 
In particular, $\{\G(4,q), \G^+(4,q), \bar{\G}(4,q) \}$ are equienergetic and non-isospectral when $|c|<\tfrac{2\sqrt{3}}{3} |d|$.
\end{enumerate}
\end{coro}

\begin{proof}
By Theorem \ref{equien Gp comp}, it is enough to see that among the non principal eigenvalues of $\G(k,q)$, $k=3,4$, exactly one is positive and the other ones are all negative. 

($i$) By Theorem \ref{gp3q}, the spectra of $\G(3,q)$ is given by
$$Spec(\G(3,q)) = \big\{ [n]^1, \big[\tfrac{a\sqrt[3]{q}-1}{3}\big]^n, \big[\tfrac{-\frac{1}{2} (a+9b)\sqrt[3]{q}-1}{3}\big]^n, 
\big[\tfrac{-\frac 12 (a-9b) \sqrt[3]{q}-1}{3}\big]^n \big\} $$
with $a,b$ satisfying \eqref{ab27}. 
Now, if $a>9|b|>0$ then $a-9b>0$ and $a+9b>0$, this implies that 
$$\tfrac{a\sqrt[3]{q}-1}{3}>0, \qquad \tfrac{-\frac{1}{2} (a+9b)\sqrt[3]{q}-1}{3}<0, \qquad \text{and} \qquad 
\tfrac{-\frac 12 (a-9b) \sqrt[3]{q}-1}{3}<0.$$
On the other hand, if $a<0$ and $-a<9|b|$ we have that either
$$\tfrac{a\sqrt[3]{q}-1}{3}<0, \qquad \tfrac{-\frac{1}{2} (a+9b)\sqrt[3]{q}-1}{3}<0, \qquad \text{and} \qquad 
\tfrac{-\frac 12 (a-9b) \sqrt[3]{q}-1}{3}>0 $$ 
for $b>0$ or else
$$\tfrac{a\sqrt[3]{q}-1}{3}<0, \qquad \tfrac{-\frac{1}{2} (a+9b)\sqrt[3]{q}-1}{3}>0, \qquad \text{and} \qquad 
\tfrac{-\frac 12 (a-9b) \sqrt[3]{q}-1}{3}<0$$ 
for $b<0$. Hence, the conditions given by $(i)$ assure that exactly one of the non-principal eigenvalues is positive.

($ii$) By Theorem \ref{gp4q}, the spectra of $\G(4,q)$ is given by
$$Spec(\G(4,q)) = \big\{ [n]^1, \big[\tfrac{\sqrt{q} + 4d\sqrt[4]{q}-1}{4}\big]^n, \big[\tfrac{\sqrt{q} - 4d\sqrt[4]{q}-1}{4}\big]^n, 
\big[\tfrac{-\sqrt{q} + 2c \sqrt[4]{q}-1}{4}\big]^n, \big[\tfrac{-\sqrt{q} - 2c \sqrt[4]{q}-1}{4}\big]^n \big\}$$
with $c,d$ satisfying conditions \eqref{c4d}.
In a similar way as in ($i$), one can check that the condition $2|c|<\sqrt[4]{q}<4|d|$ assures that one of the eigenvalues of $\G(4,q)$ is positive and the other ones are negative. Therefore, we have that
$\{\G(4,q), \G^+(4,q), \bar{\G}(4,q) \}$ are equienergetic and non-isospectral. 
For the last sentence, notice that the conditions $2|c|<\sqrt[4]{q}$ and $\sqrt[4]{q}<4|d|$ are equivalent to 
$3c^2<4d^2$ and $c^2<12 d^2$ respectively, since $\sqrt q= c^2+4d^2$. Thus, if 
$|c|<\tfrac{2\sqrt{3}}{3} |d|$ then $2|c|<\sqrt[4]{q}$ and $\sqrt[4]{q}<4|d|$. Therefore, we obtain that $\{\G(4,q), \G^+(4,q), \bar{\G}(4,q) \}$ are equienergetic and non-isospectral in this case, as asserted. 
\end{proof}

We now illustrate the previous proposition and corollary. 

\begin{exam} 
Here we give complementary equienergetic pairs of graphs $\G(k,q)$ with $k=3,4$.
	
($i$) Let $p=7$ and  $m=6$. Hence, $q=7^6=117{.}649$, $p\equiv 1 \pmod 3$ and the pair $(3,7^6)$ is not semiprimitive. The integers $a=13$ and $b=1$ satisfy \eqref{ab27} in Theorem~\ref{gp3q}, since $4 \sqrt[3]{q} = 13^2+ 27\cdot 1^2=196$.
Thus, the integers $a,b$ satisfy the hypothesis of Corollary \ref{cond equien} and hence we known that $\G(3,7^6)$ and 
$\bar\G(3,7^6)$ are equienergetic and non-isospectral, without the need to compute the spectrum. 
The spectrum of $\G(3,7^6)$ is given by
$$Spec(\G(3,7^6)) = \{ [n]^1, [212]^{n}, [-180]^{n}, [-33]^{n} \}$$
where $n= \frac{q-1}3=39216$. Therefore, we see that there is only one positive non-principal eigenvalue and Theorem~\ref{equien Gp comp} also shows that  
$\G(3,7^6)$ and $\bar\G(3,7^6)$ are equienergetic and non-isospectral.

\noindent
($ii$) Let $p=5$, $m=8$ and $q=5^8=390625$. Thus $p\equiv 1 \pmod 4$ and $(4, 5^8)$ is not a semiprimitive pair of integers. The integers $c=-7$ and $d=12$ satisfy \eqref{c4d} in Theorem \ref{gp4q}, since $625=\sqrt{q}=7^2+4\cdot 12^2$, and hence we have 
$$Spec(\G(4,5^8)) = \{ [n]^1, [456]^{n}, [-69]^{n}, [-144]^{n}, [-244]^{n} \}$$
where $n=\frac{q-1}4=97656$. 
Therefore, the spectra of $\G(3,7^6)$ has only one positive non-principal eigenvalu and by Theorem~\ref{equien Gp comp} we have that $\G(4,5^8)$ and $\bar\G(4,5^8) $ are equienergetic and non-isospectral. 
This also follows more easily using Corollary \ref{cond equien} directly, since $c$ and $d$ satisfy $|c|<2\tfrac{\sqrt 3}{3}|d|$.
\hfill $\lozenge$
\end{exam}

\begin{exam}
Now, we give a graph $\G(3,q)$ which it not complementary equienergetic.
Let $p=7$ and $m=3$, hence $q=7^3=343$. Since $p\equiv 1 \pmod 3$, we have to find integers  $a,b$ such that
$28=4 \sqrt q=a^2+27b^2$, $a\equiv 1 \pmod 3$ and $(a,7)=1$. Clearly $a=b=1$ satisfy these conditions. By ($a$) in Theorem~\ref{gp3q} we have $Spec(\G(3,7^3)) = \big\{ [114]^1, [9]^{114}, [2]^{114}, [-12]^{114} \big\}$
and thus we get $Spec(\bar \G(3,7^3)) = \{ [2\cdot 114]^1, [-10]^{114}, [-3]^{114}, [11]^{114} \}$. 
Hence,  
\begin{gather*}
E(\G(3,7^3))=114 \, (1+9+2+12)  %=24 \cdot 114 = 2{.}736 \ne 2{.}964 = 
\ne 114 \, (2+10+3+11) %=26 \cdot 114 
= E(\bar \G(3,7^3)).	
\end{gather*}
Therefore, $E(\G(3,7^3)) \ne E(\bar \G(3,7^3))$. Notice that the graph $\G(3,7^3)$ has more that one positive non-prinicpal eigenvalue. Also, the conditions for $a$ and $b$ in Corollary \ref{cond equien} do not hold.
	
Similarly, by using ($a$) in Theorem \ref{gp4q} one can get that $E(\G(4,5^4)) \ne E(\bar{\G(4,5^4)})$.
\hfill $\lozenge$
\end{exam}

\section{Infinite pairs of complementary equienergetic graphs}
In this final section we give infinite families of complementary equienergetic non-isospectral graphs 
$\{\Gamma^*(k,q), \bar \Gamma^*(k,q)\}$ for $k=3,4$, where $q$ is a power of a prime $p$, for infinite different prime numbers $p$. 
Since we will take $(k,q)$ a non-semiprimitive pair of integers, the involved graphs will be all neither bipartite nor strongly regular. This complements the results obtained in \cite{PV5} where we characterized all bipartite graphs and all strongly regular graphs which are complementary equienergetic.

We begin with GP-graphs of the form $\G(3,q)$.

\begin{thm} \label{thm Equien 3}
Let $p$ be a prime with $p\equiv 1 \pmod{3}$.
If there is some $t\in \N$ such that 
$p^t=x^2+27 y^2$ 
for some $x,y \in \Z$ with $(x,p)=1$, 
then there are infinitely many  $\ell \in \mathbb{N}$, such that
$\{\G(3,p^{3(t\ell+s)}), \bar \G(3,p^{3(t\ell+s)}) \}$ 
are equienergetic and non-isospectral for any $s\in\{0,1,\ldots,t-1\}$.
\end{thm}

\begin{proof}
Let $x,y,t$ be integers satisfying $p^t=x^2+27y^2$ with $t>0$ and $(x,p)=1$. We split the proof in two cases, 
$s=0$ and $s>0$.

($a$) First, assume that $s=0$.
Since $p>3$ is prime, we have that necessarily $x\not \equiv 0\pmod{3}$, so we can choose $x$ such that $x\equiv 1 \pmod{3}$.

If we put $z_{x,y} = x + 3\sqrt{3}i y$, then $p^t= \|z_{x,y}\|^2$. 
From the proof of Theorem \ref{spectra der 3 gen}, the spectra of $\G(3,p^{3t\ell})$ is given by
	$$\mathrm{Spec}\,\G(3,p^{3t\ell})=\big\{ [n_{\ell}]^1, \big[ \tfrac{a_{\ell}p^{t\ell}-1}{3}\big]^{n_{\ell}}, \big[ \tfrac{-\frac{1}{2} (a_{\ell}+9b_{\ell})p^{t\ell}-1}{3}\big]^{n_{\ell}}, 
	\big[ \tfrac{-\frac 12 (a_{\ell}-9b_{\ell}) p^{t\ell}-1}{3}\big]^{n} \big\}$$
with $n_\ell = \frac{p^{3t\ell}-1}{3}$ and $a_{\ell}=-2x_{\ell-1}$ and $b_{\ell}=-2 y_{\ell-1}$, where $x_{\ell}$ and $y_{\ell}$ are recursively defined by
	\begin{equation}\label{recursion t}
		x_\ell = x_0 x_{\ell-1}-27 y_0 y_{\ell-1} \qquad \text{and} \qquad y_{\ell}=x_0 y_{\ell-1}+y_0 x_{\ell-1}
	\end{equation}
for $\ell>0$, where $x_0=x, \,y_{0}=y$.
By Corollary \ref{cond equien}, the GP-graph $\G(3,p^{3t\ell})$ is equienergetic with its complement, 
in the following cases:
\begin{equation} \label{cases}
	(i) \quad \text{if $a_{\ell}>0$ and $a_{\ell} >9|b_{\ell}|$} \qquad \text{or} \qquad 
	(ii) \quad \text{if $a_{\ell}<0$ and $-a_{\ell}<9|b_{\ell}|$}.
\end{equation}
If $a_{\ell},b_{\ell}$ are both negative numbers, then item $(ii)$ written in terms of $x_{\ell},y_\ell$ is equivalent to
$x_{\ell-1}<9y_{\ell-1}$ with %\qquad \text{with} \quad 
$x_{\ell-1},y_{\ell-1}>0$.
By taking into account that $x_{\ell-1} = \mathrm{Re}(z_{x_{\ell-1},y_{\ell-1}})$ and  
$y_{\ell-1} = \frac{\sqrt{3}}{9}\mathrm{Im}(z_{x_{\ell-1},y_{\ell-1}})$  
and the fact that $z_{x_{\ell-1},y_{\ell-1}}=z_{x,y}^{\ell}$, if we take $z:=z_{x,y}$ then
	$$x_{\ell-1}=\mathrm{Re}(z^{\ell}) \quad \text{ and} \quad y_{\ell-1}=\tfrac{\sqrt{3}}{9} \mathrm{Im}(z^{\ell}).$$
So, item $(ii)$, in these terms, implies that
	\begin{equation}\label{p equien t}
		\tfrac{\pi}{6} = \arctan(\tfrac{\sqrt{3}}{3}) < \mathrm{Arg}(z^{\ell}) < \tfrac{\pi}{2} \quad \Rightarrow \quad E(\G(3,p^{3t\ell})) = E(\bar \G(3,p^{3t\ell})).
	\end{equation}

Notice that $(y_{\ell},p)=1$ for all $\ell \in \mathbb{N}_0$, since $p\mid x_{\ell}$ otherwise. 
So, we have that $\mathrm{Im}\, z^{\ell} \ne 0$ for all $\ell\in\mathbb{N}_0$. This implies that $\mathrm{Arg} \,z \not \in \mathbb{Q} \pi$, so we obtain that the classes of $\mathrm{Arg}\, z^{\ell}$ are dense in $\mathbb{R}/2\pi\mathbb{Z}$ when $\ell$ runs over $\mathbb{N}_0$. Therefore, there are infinite values $\ell \in \mathbb{N}_0$ satisfying \eqref{p equien t}, so there are infinite values $\ell\in \mathbb{N}_0$ such that $\G(3,p^{3t\ell})$ is equienergetic with its complement, as asserted.
	
\smallskip 
($b$) Now assume that $s\in \{1,\ldots,t-1\}$, and let $a_{0,s},b_{0,s}\in \mathbb{Z}$ with $a_{0,s}\equiv 1\pmod{3}$ and $(a_{0,s},p)=1$ such that 
	$$4p^s=a_{0,s}^2+27b_{0,s}^2 = \|z_{a_0,b_0}\|^2.$$
As in the proof of Theorem \ref{spectra der 3 gen}, if $z_{a_\ell,s,b_{\ell.s}}=a_{\ell,s}+3\sqrt{3} b_{\ell,s}i$ 
where $a_{\ell,s}$ and $b_{\ell,s}$ are defined recursively by 
	$$a_{\ell,s}=a_{0,s} x_{\ell-1}-27 b_{0,s} y_{\ell-1} \qquad \text{and}\qquad 
	b_{\ell,s}=a_{0,s} y_{\ell-1}+b_{0,s} x_{\ell-1},$$
then $4p^{t\ell+s}=a_{\ell,s}^2+27b_{\ell,s}^2=\|z_{a_{\ell,s},b_{\ell,s}}\|^2$ with $a_{\ell,s}\equiv 1\pmod{3}$ and $(a_{\ell,s},p)=1$, thus $a_{\ell,s}$, $b_{\ell,s}$ give the spectra of $\G(3,p^{3(t\ell+s)})$.
Moreover, we have that 
	$$z_{a_{\ell,s},b_{\ell,s}}=z_{a_{0,s},b_{0,s}} z_{x_0,y_0}^{\ell}.$$
As before, by Corollary \ref{cond equien} the GP-graph $\G(3,p^{3(t\ell+s)})$ is equienergetic with its complement in the two cases in \eqref{cases}, with $a_\ell=a_{\ell,s}$ and $b_\ell = b_{\ell,s}$.   
In this case, when $a_{\ell,s},b_{\ell,s}$ are both positive integers, the item ($i$) ensures that 
	\begin{equation}\label{p equien t s}
		0< \mathrm{Arg} (z_{a_{\ell,s},b_{\ell,s}}) < \arctan(\tfrac{\sqrt{3}}{3}) = \tfrac{\pi}{6} \quad \Rightarrow \quad  E(\G(3,p^{3(t\ell+s)})) = E(\bar \G(3,p^{3(t\ell+s)})).
	\end{equation}
Since $\mathrm{Arg} (z_{x_0,y_0}^{\ell})$ is dense in $\mathbb{R}/2\pi\mathbb{Z}$ when $\ell$ runs over $\mathbb{N}$ and $z_{a_{\ell,s},b_{\ell,s}}=z_{a_{0,s},b_{0,s}} z_{x_0,y_0}^{\ell}$, hence we obtain that
$\mathrm{Arg} (z_{a_{\ell,s},b_{\ell,s}})$ is dense in $\mathbb{R}/2\pi\mathbb{Z}$ when $\ell$ runs over $\mathbb{N}$. 
Therefore, there are infinite values $\ell \in \mathbb{N}$ satisfying \eqref{p equien t s}, so there are infinite values $\ell\in \mathbb{N}$ such that $\G(3,p^{3(t\ell+s)})$ is equienergetic with its complement, as desired.	
\end{proof}

As a direct consequence of the previous theorem we obtain the following result.
\begin{coro}
	Let $p$ be a prime with $p\equiv 1 \pmod{3}$ such that $2$ is a cubic residue modulo $p$. 
	Then, there are infinite $\ell \in \mathbb{N}$, such that $\{\G(3,p^{3\ell}),\bar \G(3,p^{3\ell})\}$ are equienergetic.
\end{coro}

\begin{proof}
Take $t=1$. If $2$ is a cubic residue modulo $p$ then $p=X^2+27Y^2$ has integer solutions $x,y$ with $(x,p)=1$, by \eqref{2res conds}. Now, taking $s=0$ in the previous theorem we get the desired result.
\end{proof}

We now give the analogous of Theorem \ref{thm Equien 3} for GP-graphs of the form $\G(4,q)$.

\begin{thm} \label{thm Equien 4}
	Let $p$ be a prime with $p\equiv 1\pmod{4}$. Then, there exist infinitely many $\ell\in \mathbb{N}$ such that
	$\{\G(4,p^{4\ell}), \bar \G(4,p^{4\ell})\}$ are equienergetic.
\end{thm}

\begin{proof}
If we put $z_{x,y}=x+2yi $, then $p^2= \|z_{c_1,d_1}\|^2$,
where $c_1$ and $d_1$ are integral solutions of $p^2=X^2+ 4Y^2$ with $c_1 \equiv 1 \pmod 4$ and $(c_1,p)=1$. 
From the proof of Theorem \ref{spectra der 4}, the spectra of $\G(3,p^{4\ell})$ is given by
	$$Spec(\G(4,p^{4\ell})) = \big\{ [n_{\ell}]^1, \big[\tfrac{p^{2\ell} + 4d_{\ell}p^{\ell}-1}{4}\big]^{n_{\ell}}, \big[\tfrac{p^{2\ell} - 4d_{\ell}p^{\ell}-1}{4}\big]^{n_{\ell}}, 
	\big[\tfrac{-p^{2\ell} + 2c_{\ell} p^{\ell}-1}{4}\big]^{n_{\ell}}, \big[\tfrac{-p^{2\ell} - 2c_{\ell} p^{\ell}-1}{4}\big]^{n_{\ell}}\big\},$$
where $n_\ell=\frac{p^{4\ell}-1}4$ and $c_\ell, d_{\ell}$ are recursively defined as follows
	\begin{equation*}
		c_{\ell+1}=c_1 c_{\ell}-4d_{1} d_{\ell} \qquad \text{and} \qquad d_{\ell+1}= c_1 d_{\ell}+d_1 c_{\ell}.
	\end{equation*}	
By Corollary \ref{cond equien}, the graph $\G(4,p^{3\ell})$ is equienergetic with its complement if it satisfies
$|c_{\ell}|<\tfrac{2\sqrt{3}}{3} |d_{\ell}|$.
If $c_{\ell}$ and $d_{\ell}$ are both positive integers, using that $z_{c_{\ell},d_{\ell}}=z_{c_1,d_1}^{\ell}$, we obtain that
	\begin{equation}\label{p equien 4}
		\tfrac{\pi}{3} = \arctan(\sqrt{3}) < \mathrm{Arg}(z_{c_1,d_{1}}^{\ell}) < \tfrac{\pi}{2} \qquad \Rightarrow \qquad 
		E(\G(4,p^{4\ell})) = E(\bar \G(4,p^{4\ell})) .
	\end{equation}	
The same argument as in the case $s=0$ in Theorem \ref{thm Equien 3} ensures that the classes of $\mathrm{Arg}(z_{c_1,d_1}^{\ell})$ are dense in $\mathbb{R}/2\pi\mathbb{Z}$ when $\ell$ runs over $\mathbb{N}$. Therefore, there are infinite values $\ell \in \mathbb{N}$ satisfying \eqref{p equien 4}, so there are infinite values $\ell\in \mathbb{N}$ such that $\G(4,p^{4\ell})$ is equienergetic with its complement, as we wanted to see.
\end{proof}

As a final remark, we want to stress that in \cite{PV5} we have classified all bipartite regular graphs $\Gamma_{bip}$ and all strongly regular graphs $\Gamma_{srg}$ which are complementary equienergetic, i.e.\@ $\{\Gamma_{bip}, \bar{\Gamma}_{bip}\}$ and $\{\Gamma_{srg}, \bar{\Gamma}_{srg}\}$ are equienergetic pairs of non-isospectral graphs. 
Here, using non-semiprimitive GP-graphs of the form $\G(3,q)$ and $\G(4,q)$, Theorems \ref{thm Equien 3} and \ref{thm Equien 4} ensure that there are infinitely many pairs of equienergetic non-isospectral regular graphs $\{\Gamma, \bar \Gamma\}$ which are neither bipartite nor strongly regular.


\begin{thebibliography}{XXX}
	

\bibitem{Ba} \textsc{R.\@ Balakrishnan}. 
\textit{The energy of a graph}, Linear Algebra Appl.\@ 387 (2004), 287--295.
 
\bibitem{BH} \textsc{A.\@ E.\@ Brouwer, W.\@ H.\@ Haemers}. 
\textit{Structure and uniqueness of the $(81,20,1,6)$ strongly regular graph.}
Discr.\@ Math.\@ 106/107 (1992) 77--82. %(pp. 130, 143, 209)

\bibitem{C} \textsc{D.A.\@ Cox}. \textit{Primes of the form $x^2+n y^2$}. 
John Wiley \& Sons, Inc.  1989.


\bibitem{Ca} \textsc{P.J.\@ Cameron}.
	\textit{Strongly regular graphs}. Chapter 12 in Selected Topics in Graph Theory, pp.\@ 337--360, L.\@ Beineke and R.\@ Wilson (Eds.), Academic Press, NewYork, 1978.
	
\bibitem{Ca2} \textsc{P.J.\@ Cameron, J.-M.\@ Goethals, J.J.\@ Seidel}. 
	\textit{Strongly regular graphs with strongly regular subconstituents}. J.\@ Algebra 55 (1978), 257--280.

\bibitem{Ca3} \textsc{P.J.\@ Cameron}. 
	\textit{6-transitive graphs}. J.\@ Combinatorial Theory (B) 28 (1980), 168--179.
	
\bibitem{CDS} \textsc{D.M.\@ Cvetkovi\'c, M.\@  Doob, H.\@ Sachs}. \textit{Spectra of graphs. Theory and applications}. Third edition. Johann Ambrosius Barth, Heidelberg, 1995.

\bibitem{DY}{\sc C.\@ Ding, J.\@ Yang}. 
\textit{Hamming weights in irreducible cyclic codes}.
Discrete Math.\@  313:4 (2013), 434--446.

\bibitem{GPI} \textsc{H.A.\@ Ganie, S.\@ Pirzada, A.\@ Iv\'anyi}.
\textit{Energy, Laplacian energy of double graphs and new families of equienergetic graphs}.
Acta Universitatis Sapientiae, Informatica \textbf{6:1} (2014), 89--116.

\bibitem{GK}{\sc  D.\@ Ghinelli, J.D.\@ Key}. 
\textit{Codes from incidence matrices and line graphs of Paley graphs}.
Adv.\@ Math.\@ Comm.\@ \textbf{5} (2011) 93--108.

\bibitem{Gu} \textsc{I.\@ Gutman}. 
\textit{The energy of a graph: old and new results}, in Algebraic Combinatorics
and Applications, A.\@ Betten, A.\@ Kohner, R.\@ Laue, and A.\@ Wassermann, eds., Springer,
Berlin, 2001, 196--211.

\bibitem{HX} \textsc{Y.\@ Hou, L.\@ Xu}. 
\textit{Equienergetic bipartite graphs}. 
MATCH Commun.\@ Math.\@ Comput.\@ Chem.\@ \textbf{57} (2007), 363--370.

\bibitem{J}{\sc G.\@ Jones}.
\textit{Characterisations and Galois conjugacy of generalised Paley maps.}. 
J.\@ Comb.\@ Theory, Ser.\@ B \textbf{103:2} (2013) 209--219.

\bibitem{JW}{\sc G.\@ Jones, J.\@ Wolfart}.
\textit{Dessins d'Enfants on Riemann Surfaces}. 
Springer International Publishing Switzerland, (2016).

\bibitem{Il}\textsc{A.\@ Ili\'c}. 
\textit{The energy of unitary Cayley graphs}. 
Linear Algebra Appl.\@ 431 (2009) 1881--1889.

\bibitem{KL}{\sc J.D.\@ Key, J.\@ Limbupasiriporn}. 
\textit{Partial permutation decoding for codes from Paley graphs}.
Cong.\@ Numer.\@ \textbf{170} (2004) 143--155.

\bibitem{LSG} \textsc{X.\@ Li, Y.\@ Shi, I.\@ Gutman}. 
\textit{Graph energy}. Springer, New York, 2012.

\bibitem{LP}{\sc T.K.\@ Lim, C.\@ Praeger}. 
\textit{On Generalised Paley Graphs and their automorphism groups}.
Michigan Math.\@ J.\@ \textbf{58} (2009) 294--308.

\bibitem{LZ2} 
\textsc{X.\@ Liu, S.\@ Zhou}. 
\textit{Eigenvalues of Cayley graphs}, 2019, arXiv:1809.09829v2.

\bibitem{PP}{\sc G.\@ Pearce, C.\@ Praeger}. 
\textit{Generalised Paley graphs with a product structure}. 
Ann.\@ Comb.\@ \textbf{23} (2019) 171--182.

\bibitem{PV2} \textsc{R.A.\@ Podest\'a, D.E.\@ Videla}.
\textit{Spectral properties of generalized Paley graphs and their associated irreducible cyclic codes},
20 pages, arXiv:1908.08097 (2019).

\bibitem{PV3} \textsc{R.A.\@ Podest\'a, D.E.\@ Videla}.
\textit{Integral equienergetic non-isospectral unitary Cayley graphs},
Linear Algebra and its Applications \textbf{612} (2021), 42--74.

\bibitem{PV4} \textsc{R.A.\@ Podestá, D.E.\@ Videla}. 
\textit{The weight distribution of irreducible cyclic codes associated with descomposable generalized Paley graphs}. 
Adv. Math. Commun. (2021) \textit{Online first.} doi: 10.3934/amc.2021002.

\bibitem{PV5} \textsc{R.A.\@ Podest\'a, D.E.\@ Videla}.
\textit{On regular graphs equienergetic with their complements}.
Linear and Multilinear Algebra, in press (2022) doi: 10.1080/03081087.2022.2033152

\bibitem{PV6} \textsc{R.A.\@ Podest\'a, D.E.\@ Videla}.
\textit{The Waring's problem over finite fields through generalized Paley graphs}.
Discrete Math.\@ \textbf{344} (2021) 112324.

\bibitem{PV7} \textsc{R.A.\@ Podest\'a, D.E.\@ Videla}.
\textit{A reduction formula for Waring numbers through generalized Paley graphs},
25 pages, arXiv: 1911.12761 (2021)

\bibitem{Ra} \textsc{H.S.\@ Ramane, H.B.\@ Walikar, S.B.\@ Rao, B.D.\@ Acharya, P.R.\@ Hampiholi, S.R.\@ Jog, I.\@ Gutman}. 
\textit{Equienergetic graphs}. 
Kragujevac J.\@ Math.\@ 26 (2004), 5--13. 

\bibitem{RPPAG} \textsc{H.S.\@ Ramane,  K.\@ Ashoka, B.\@ Parvathalu, D.D.\@ Patil, I.\@ Gutman}. 
\textit{On complementary equienergetic strongly regular graphs}. 
Discrete Math.\@ Lett.\@ 4 (2020), 50--55.

\bibitem{SS} \textsc{C.\@ Schneider, A.C.\@ Silva}.
\textit{Cliques and colorings in generalized Paley graphs and an approach to synchronization}.
J.\@ Algebra Appl.\@ 14:6 (2015) 1550088.

\bibitem{SL}{\sc P.\@ Seneviratne, J.\@ Limbupasiriporn}. 
\textit{Permutation decoding from generalized Paley graphs}. 
Appl.\@ Algebra in Eng.\@ Comm.\@ and Computing \textbf{24} (2013) 225--236.

\bibitem{V} \textsc{D.E.\@ Videla}.
\textit{On diagonal equations over finite fields via walks in NEPS of graphs}.
Finite Fields App.\@ \textbf{75} (2021) 101882.

\bibitem{Y1} \textsc{C.H.\@ Yip}.
\textit{On the directions determined by cartesian products and the clique number of generalized Paley graphs}.
Integers.\@ \textbf{21}, Paper A51 (2021).

\bibitem{Y2} \textsc{C.H.\@ Yip}.
\textit{On the clique number of Paley graphs of prime power order}.
Finite Fields App.\@ \textbf{77} (2022) 101930.

\end{thebibliography}
\end{document}